\documentclass[12pt]{amsart}
\usepackage{amsmath,amssymb,xypic,array}
\usepackage[T1]{fontenc}
\usepackage{amsfonts}
\usepackage{amsmath}
\usepackage{amssymb}
\usepackage{amsthm}
\usepackage{setspace}
\usepackage[cp850]{inputenc}
\usepackage{hyperref}
\usepackage{graphicx, epsfig, psfrag}
\usepackage{fancyhdr}

\usepackage{pgf}
\usepackage{tikz}
\usepackage{booktabs}
\usepackage{longtable}
\usepackage{geometry}
\usetikzlibrary{trees,arrows,automata}

\providecommand{\U}[1]{\protect\rule{.1in}{.1in}}
\providecommand{\U}[1]{\protect\rule{.1in}{.1in}}
\providecommand{\U}[1]{\protect\rule{.1in}{.1in}}
\providecommand{\U}[1]{\protect\rule{.1in}{.1in}}
\providecommand{\U}[1]{\protect\rule{.1in}{.1in}}
\input{xy}
\xyoption{all}
\theoremstyle{theorem}

\newtheorem{Theorem}{Theorem}[section]

\newtheorem*{question}{Question}

\newtheorem{remark}{Remark}[section]

\newtheorem{Lemma}[Theorem]{Lemma}
\newtheorem{Proposition}[Theorem]{Proposition}
\newtheorem{Corollary}[Theorem]{Corollary}

\theoremstyle{definition}

\newtheorem{Definition}[Theorem]{Definition}
\newtheorem{Remark}[Theorem]{Remark}

\newtheorem{Example}[Theorem]{Example}

\numberwithin{equation}{section}

\def\leq{\leqslant}
\def\geq{\geqslant}

\def\phi{\varphi}
\def\ro[#1]{{\textcolor{red}{#1}}}

\title[Stochastic adding machines based on Bratteli diagrams] %Use the shortened version of the full title
{Stochastic adding machines based on Bratteli diagrams}

\author[D. A. Caprio, A. Messaoudi and G. Valle]{}

\subjclass{Primary: 37A30, 37F50; Secondary: 60J10, 47A10.}
\keywords{Markov chains, stochastic Vershik map, Bratteli diagrams, spectrum of transition operators, fibered Julia sets.}

\email{$^*$danilo.caprio@unesp.br}
\email{$^{\dag}$messaoud@ibilce.unesp.br}
\email{$^{\ddag}$glauco.valle@im.ufrj.br}

\thanks{$^*$ Supported by FAPESP grant $2015/26161-6$}

\thanks{$^{\dag}$ Supported by CNPq grant $307776/2015-8$ and FAPESP project $2013/23643-4$}

\thanks{$^{\ddag}$ Supported by FAPERJ grants $E-26/203.048/2016$ and CNPq grants $305805/2015-0$ and $421383/2016-0$}

\begin{document}
\maketitle

% Enter the first author's name and address:
\centerline{\scshape Danilo Antonio Caprio$^*$}
\medskip
{\footnotesize
	% please put the address of the first author
	\centerline{UNESP - Departamento de Matem\'atica do Instituto de
		Bioci\^encias, Letras e Ci\^encias Exatas.}
	\centerline{Rua Crist\'ov\~ao Colombo, 2265, Jardim Nazareth, 15054-000 S\~ao Jos\'e do Rio Preto, SP, Brasil.}
} % Do not forget to end the {\footnotesize by the sign }

\medskip

\centerline{\scshape Ali Messaoudi$^\dag$}
\medskip
{\footnotesize
	% please put the address of the first author
	\centerline{UNESP - Departamento de Matem\'atica do Instituto de
		Bioci\^encias, Letras e Ci\^encias Exatas.}
	\centerline{Rua Crist\'ov\~ao Colombo, 2265, Jardim Nazareth, 15054-000 S\~ao Jos\'e do Rio Preto, SP, Brasil.}
} % Do not forget to end the {\footnotesize by the sign }

\medskip

\centerline{\scshape Glauco Valle$^\ddag$}
\medskip
{\footnotesize
	% please put the address of the first author
	\centerline{Universidade Federal do Rio de Janeiro - Instituto de Matem\'atica.}
	\centerline{Caixa Postal 68530, cep 21945-970, Rio de Janeiro, Brasil.}
} % Do not forget to end the {\footnotesize by the sign }

\bigskip

\begin{abstract}
In this paper, we define some Markov Chains associated to Vershik maps on Bratteli diagrams. We study probabilistic and spectral properties of their transition operators and we prove that the spectra of these operators are connected to Julia sets in higher dimensions. We also study topological properties of these spectra.
\end{abstract}

\section{Introduction}
\label{sec:intro}
Let $g$ be a holomorphic map on $\mathbb{C}^d$, where $d\geq 1$ is an integer. The set $K(g)$ of $z\in\mathbb{C}^d$ such that the forward orbit $\{g^n(z): n \in\mathbb{N}\}$ is bounded is called the (d-dimensional) filled Julia set of $g$. Filled Julia sets and their boundaries (called Julia sets) were defined independently by Julia and Fatou (\cite{fatou19} and \cite{fatou21}, \cite{julia18} and \cite{julia22}).

The study of Julia sets is connected to many areas of mathematics as dynamical systems, complex analysis, functional analysis and number theory, among others (see for example \cite{BC}, \cite{Cr}, \cite{DV}, \cite{Do}, \cite{DH}, \cite{Fa}, \cite{FS}, \cite{GS}, \cite{K}, \cite{L}, \cite{m}, \cite{MT}, \cite{Y}).

There is an important connection between Julia sets and stochastic adding machines. A first example was given by Killeen and Taylor in \cite{kt1} as follows: let $n$ be a nonnegative integer, then it can be written in a unique way in base $2$ as $n=\sum_{i=0}^{k}\varepsilon_i(n) 2^i=\varepsilon_k\ldots \varepsilon_0$, for some $k\ge 0$, where $\varepsilon_k = 1$ and $\varepsilon_i\in\{0,1\}$, for all $i\in\{0,\ldots,k-1\}$. It is known that the addition
of $1$ is given by a classical algorithm, namely $n+1 = \varepsilon_k\ldots\varepsilon_{l+1}(\varepsilon_l+1)0\ldots 0$
where $l=\min\{i\geq 0:\varepsilon_i(n)=0\}$. Killeen and Taylor defined the stochastic
adding machine assuming that each time a carry should be added, it is
added with probability $0 < p < 1$ and it is not added with probability $1-p$.
Moreover, the algorithm stops when the first carry is not added. So this random algorithm maps
$n = \varepsilon_k\ldots \varepsilon_0$ to $n$ itself with probability $1-p$, to $n+1$ with probability $p^{l+1}$ and to $m=n-2^r+1=\varepsilon_k\ldots \varepsilon_{d+1}\ldots \varepsilon_r0\ldots 0$ with probability $p^r(1-p)$. With this they obtained a countable Markov chain whose associated transition operator $S=(p_{i,j})_{i,j\in\mathbb{N}}$ is a bistochastic infinite matrix whose spectrum is equal to the filled Julia set of the quadratic map $\frac{z^2-(1-p)}{p} \, , \ z \in \mathbb{C}$.

In \cite{msv}, \cite{ms}, \cite{mu} and \cite{mv}, stochastic adding machines based on other systems of numeration have been introduced. They are connected to one-dimensional fibered Julia sets (see \cite{msv}) and also to Julia sets in dimension greater than one (\cite{cap}, \cite{ms} and \cite{mu}). A d-dimensional fibered filled Julia set of a sequence $(g_j)_{j\ge 1}$ of holomorphic maps on $\mathbb{C}^d$ is the set $K((g_j)_{j\ge 1})$ of $z \in\mathbb{C}^d$ such that the forward orbit $\{\tilde{g}_j(z): j\in\mathbb{N}\}$ is bounded, where $\tilde{g}_j = g_j \circ g_{j-1} \circ ... \circ g_1$ for all $j\geq 1$.

In this paper, we introduce stochastic adding machines associated to Vershik maps on Bratteli diagrams. Bratteli diagrams are important objects in the theories of operator algebras and dynamical systems. It was originally defined in 1972 by O. Bratteli \cite{brat} for classification of C${}^*$-algebras. Bratteli diagrams turned out to be a powerful tool in the study of measurable, Borel, and Cantor dynamics (see \cite{GPS}, \cite{HPS}, \cite{m}, \cite{Ver}). The interest on Bratteli diagrams is that any aperiodic transformation in measurable, Borel, and Cantor dynamics can be realized as a Vershik map acting on the path space of a Bratteli diagram (see \cite{BDK}, \cite{HPS}, \cite{m}, \cite{Ver}, \cite{Ver2}).

A particular application arises when we use the Vershik map to embed $\mathbb{Z}_+$ into the set of paths of the associated Bratteli diagram. This embedding allows us to consider the restriction of the Vershik map on that copy of $\mathbb{Z}_+$ as the map $n \mapsto n+1$. It also allows a representation of systems of numeration through Bratteli diagrams, making possible for us to introduce more general stochastic adding machines. Indeed we are able to define a more general Markov process on the set $X$ of infinite paths on the Bratteli diagram whose restriction to the copy of $\mathbb{Z}_+$ is the stochastic adding machine, we call this process the "Bratteli-Vershik process" or simply BV process and the associated Stochastic adding machine the Bratteli-Vershik stochastic adding machine or simply BV stochastic adding machine. 

We will give necessary and sufficient conditions that assure transience or recurrence of the BV stochastic adding machines.  We will also prove that the spectrum of the BV stochastic adding machine transition operator $S$ (acting on $l^{\infty}(\mathbb{N})$) is related to fibered filled Julia sets in higher dimension. For example, if the Bratteli diagram is stationary and its incidence matrix is $M=\left(\begin{array}{cc} a & b \\ c & d \end{array}\right)$ where $a,b,c,d$ are nonnegative integers, then the point spectrum of the transition operator of the Bratteli-Vershik stochastic adding machine associated to $M$ is related to the Julia set $$\mathcal{K}:=\{(x,y)\in\mathbb{C}^2: (g_n\circ\ldots\circ g_1(x,y))_{n\geq 1} \textrm{ is bounded} \},$$ where $g_n(x,y)=\left(\frac{1}{p_{n+1}}x^ay^b-\frac{1-p_{n+1}}{p_{n+1}}, \frac{1}{p_{n+1}}x^cy^d-\frac{1-p_{n+1}}{p_{n+1}}\right)$ and $0<p_{n+1}<1$, for all $n\geq 1$.

Just to mention an important connection, the study of these spectra gives information about the dynamical properties of transition operators
acting on separable Banach spaces (see for instance \cite{BM} and \cite{GEM}). For example, if T is topologically transitive, then any connected component of the spectrum intersects the unit circle. However, here we do not aim at the study of the dynamical properties of the transition operators. We will also study topological properties of this spectrum.

The paper is organized as follows. In Section \ref{sec:bratteli} we give a background about Bratteli diagrams and we define the Vershik map.
In Section \ref{sec:BVP} we define the BV processes and the BV stochastic adding machines giving necessary and sufficient conditions for transience, null recurrence and positive recurrence. Section \ref{sec:spectrum} is devoted to provide an exact description of the spectra of the transition operators of BV stochastic machines acting on $l^\infty(\mathbb{N})$ in the case of $2\times 2$ Bratteli diagrams. Furthermore, we prove some topological properties of this spectrum. Section \ref{generalization} describes generalization to $l\times l$, $l\geq 3$, Bratteli diagrams. 

\section{Bratteli diagrams}
\label{sec:bratteli}

\subsection{Basics on Bratteli diagrams}

In this section we introduce the necessary notation on Bratteli diagrams. Here we follow \cite{Du} and \cite{bmn} and we recommend both texts, as well as \cite{BK}, as references on Bratteli diagrams for the interested reader. 

A \emph{Bratteli diagram} is an infinite directed graph $(V,E)$ such that the vertex set $V = \bigcup_{k=0}^\infty V(k)$ 
and the edge set $E = \bigcup_{k=1}^\infty E(k)$  are partitioned into finite disjoint subsets $V(k)$ and $E(k) $, where there exist maps $s:E\longrightarrow V$ and $r:E\longrightarrow V$ such that $s$ restricted to $E(k)$ is a surjective map from $E(k)$ to $V(k-1)$ and $r$ restricted to $E(k)$ is a surjective map from $E(k)$ to $V(k)$ for every $k \ge 1$.

For every $e \in E$, $s(e)$ and $r(e)$ are called respectively the \emph{source} and \emph{range} of $e$. For convenience if $\# V(k) = l$ we denote $V(k) = \{(k,1),...,(k,l)\}$ or simply $V(k) = \{1,...,l\}$ when there is no possibility of misidentification of the value of $k$.

\begin{Remark}
It is usual to define the Bratteli diagrams under the condition that $V(0)$ is a singleton. We do not impose this condition. Our definition is more suitable to the understanding of stationarity and is more appropriate for the discussion of the results in this paper. However we could also use that condition in the definition without prejudice to the results in this paper. 
\end{Remark}

It is convenient to give a diagrammatic representation of a Bratteli diagram considering $V(k)$ as a "horizontal" level $k$ and the edges in $E(k)$ heading downwards from vertices at level $k-1$ to vertices at level $k$. Also, if $\# V(k - 1) = l(k - 1)$ and $\# V(k) = l(k)$, then $E(k)$ determines a $l(k)\times l(k - 1)$ \emph{incidence matrix} $M(k)$, where $M(k)_{i,j}$ is the number of the edges going from vertex $j$ in $V(k-1)$ to vertex $i$ in $V(k)$. By definition of Bratteli diagrams, we have that $M(k)$ has non identically zero rows and columns. 

Let $k, \tilde{k} \in\mathbb{Z}_+$ with $k<\tilde{k}$ and let $E(k+1)\circ E(k+2)\circ \ldots \circ E(\tilde{k})$ denote
the set of paths from $V (k)$ to $V (\tilde{k})$. Specifically, $E(k+1)\circ \ldots \circ E(\tilde{k})$ denotes the following set:
\begin{center}
$\{(e_{k+1},\ldots,e_{\tilde{k}}): e_i\in E(i), \textrm{ } k+1\leq i\leq \tilde{k}, \textrm{ } r(e_i)=s(e_{i+1}), \textrm{ }k+1\leq i\leq \tilde{k}-1\}$.
\end{center}
The incidence matrix of $E(k+1)\circ\ldots\circ E(\tilde{k})$ is the product $M(\tilde{k})\cdot\ldots\cdot M(k+1)$. We define $r(e_{k+1},\ldots,e_{\tilde{k}}):=r(e_{\tilde{k}})$ and $s(e_{k+1},\ldots,e_{\tilde{k}}):=s(e_{k+1})$.

In this paper, we will assume that $(V, E)$ is a \emph{simple Bratteli diagram}, i.e. for each nonnegative integer $k$, there exists an integer $\tilde{k} > k$ such that the product $M(\tilde{k})\cdot\ldots\cdot M(k+1)$ have only nonzero entries.

\subsection{Ordered Bratteli diagrams}

An ordered Bratteli diagram $(V, E,\geq)$ is a Bratteli diagram $(V, E)$ together with a partial order $\geq$ on $E$ such that edges $e,e' \in E$ are comparable if and only if $r(e) = r(e')$, in other words, we have a linear order on the set $r^{-1}(\{v\})$ for each $v \in V \setminus V(0)$.

\begin{remark} \label{rem:edgenot}
Edges in an ordered Bratteli diagram $(V, E,\geq)$ are uniquely determined by a four dimensional vector $e = (k,s,m,r)$, where $k$ means that $e \in E(k)$, $s=s(e)$ and $r=r(e)$ are the source and range of $e$ as previously defined and $m \in \mathbb{Z}_+$ is the order index means that $e=e^m \in r^{-1}(r(e)) = \{ e^0 < e^1 < ... < e^{r-1} \}$. Usually we will write $e = e_k = (s,m,r)$ carrying the level index $k$ as a subscript or suppressing it when there is no doubt about the level. 
\end{remark}

Note that if $(V,E,\geq)$ is an ordered Bratteli diagram and $k < \tilde{k}$ in $\mathbb{Z}_+$, then the set $E(k+1)\circ E(k+2)\circ\ldots\circ E(\tilde{k})$ of paths from $V(k)$ to $V(\tilde{k})$ may be given an induced order as follows:

\begin{center}
$(e_{k+1},e_{k+2},\ldots,e_{\tilde{k}}) > (e'_{k+1},e'_{k+2},\ldots,e'_{\tilde{k}})$
\end{center}
if and only if for some $i$ with $k+1\leq i\leq \tilde{k}$, $e_i > e'_i$ and $e_j =e'_j$ for $i < j \leq \tilde{k}$. 

A Bratteli diagram $(V,E)$ is \emph{stationary} if there exists $l$ such that $l =\#V (k)$ for all $k$, and (by an appropriate relabelling of the vertices if necessary) the incidence matrices between level $k$ and $k + 1$ are the same $l \times l$ matrix $M$ for all $k \ge 1$. In other words, beyond level $1$ the diagram repeats itself. 
An ordered Bratteli diagram $B = (V, E, \geq)$ is \emph{stationary} if $(V, E)$ is stationary, and the ordering on the edges with range $(k,i)$ is the same as the ordering on the edges with range $(\tilde{k},i)$ for $k,\tilde{k} \ge 2$ and $i = 1,\ldots ,l$, i.e. beyond level $1$ the diagram with the ordering repeats itself.
Furthermore, we say that $\ge$ is a \emph{consecutive ordering} if for all edges $e \le f \le e'$ with $s(e) = s(e')$ we have $s(f) = s(e) = s(e')$. To every ordered Bratteli diagram with consecutive ordering $B = (V, E, \geq)$ we associate a sequence of matrices $(Q(k))_{k\ge 1}$ called the \emph{ordering matrices} such that
\begin{enumerate}
\item[(i)] $Q(k)$ is a $(l(k)) \times (l(k-1))$ matrix;
\item[(ii)] $Q(k)_{i,j} = 0$ if and only if $M(k)_{i,j} = 0$;
\item[(iii)] The nonzero entries in each row $i$ of $Q(k)$ form a permutation in $\#\{ j : M(k)_{i,j} > 0 \}$ letters. So the row $i$ in $Q(k)$ indicates how edges inciding on vertex $i \in V(k)$ are ordered with respect to its sources in $V(k-1)$.   
\end{enumerate} 
The consecutive ordering is said to be \emph{canonical} if each row of $Q(k)$, $k\ge 1$, the permutation in $\#\{ j : M(k)_{i,j} > 0 \}$ letters is the identity.

For a stationary ordered Bratteli diagram, the consecutive ordering is also stationary, i.e $Q = Q(k)$ for every $k$. As an example consider a stationary ordered Bratteli diagram with $l=2$ and incidence matrix
$$
M=\left(\begin{array}{cc} a & b \\ c & 0 \end{array}\right) \, ,
$$
\noindent with $abc>0$. We have two possible consecutive orderings relative to the ordering matrices
$$
\left(\begin{array}{cc} 1 & 2 \\ 1 & 0 \end{array}\right) \qquad or \qquad 
\left(\begin{array}{cc} 2 & 1 \\ 1 & 0 \end{array}\right) \, ,
$$
where the first one is associated to the canonical consecutive ordering.

\subsection{The Vershik map}

Let $B = (V,E,\geq)$ be an ordered Bratteli diagram. Let $X_B$ denote the associated infinite path space, i.e.

\begin{center}
$X_B = \{(e_1,e_2,\ldots): e_i\in E(i) \textrm{ and } r(e_i) = s(e_{i+1}), \textrm{ for all } i\geq 1\}$ .
\end{center}

Under the hypotheses of the definition of a Bratteli diagram, $X_B$ is nonempty. However $X_B$ can be a finite set, this only occurs in trivial cases and do not occur for general classes of Bratteli diagrams as for instance simple Bratteli diagrams with $\# E(k) >1$ for infinitely many $k\ge 1$. Hence we require that $X_B$ is infinite for all Bratteli diagrams considered here.

We endow $X_B$ with a topology such that a basis of open sets is given by the family of cylinder sets
\begin{center}
$[e_1,e_2,\ldots,e_k]_B = \{(f_1,f_2,\ldots)\in X_B : f_i = e_i, \textrm{ for all } 1\leq i\leq k\}$ .
\end{center}
Each $[e_1,\ldots,e_k]$ is also closed, as is easily seen. Endowed with this topology, we call $X_B$ the Bratteli compactum associated with $B = (V,E,\geq)$. Let $d_B$ be the distance on $X_B$ defined by $d_B((e_j)_j ,(f_j )_j )= \frac{1}{2^k}$ where $k = \inf\{i\geq 1: e_i\neq f_i\}$. The topology of the cylinder sets coincide with the topology induced by $d_B$.

If $(V,E)$ is a simple Bratteli diagram, then $X_B$ has no isolated points, and so is a Cantor space (see \cite{m}).

Two paths in $X_B$ are said to be \emph{cofinal} if they have the same tails, i.e. the edges agree from a certain level.

Let $x = (e_1,e_2,\ldots)$ be an element of $X_B$. We will call $e_k = e_k(x)$ the k-th label of $x$. Recall from Remark \ref{rem:edgenot} that $e_k = (s_k,m_k,r_k)$ such that $r_k = s_{k+1} \in V(k)$ for every $k\ge 1$.  We let $X_B^{max}$ denote those elements $x$ of $X_B$ such that $e_k(x)$ is a maximal edge for all $k$ and $X_B^{min}$ the analogous set for the minimal edges. It is clear that from any vertex at level $k$ there is an upward maximal path to level $0$, using this we have that $X_B^{max}$ is the intersection of nonempty compact sets, so it is nonempty. Analogously $X_B^{min}$ is nonempty.

From now on we denote
$$
X^0_B := X_B\setminus{X_B^{\max}} \, .
$$

If $B=(V,E,\geq)$ is an ordered Bratteli diagram then it is easy to check that every infinite path $x\in X^0_B$ has an unique \emph{successor}. Indeed let $x = (e_1,e_2,...)\in X^0_B$ and $\zeta(x)$ be the smallest number such that $e_\zeta$ is not a maximal edge. Let $f_\zeta = f_\zeta(x)$ be the successor of $e_\zeta$ (and so $r(e_\zeta) = r(f_\zeta)$). Then the successor of $x$ is $V_B(x) = y = (f_1,\ldots,f_{\zeta-1},f_\zeta,e_{\zeta+1},...)$, where $(f_1,\ldots,f_{\zeta-1}) = (f_1(x),\ldots,f_{\zeta-1}(x)) $ is the minimal path in $E(1)\circ E(2)\circ\ldots\circ E(\zeta-1)$ with range equal to $s(f_\zeta)$, i.e. $r(f_1,\ldots,f_{\zeta-1})=s(f_\zeta)$.
Thus, it is convenient to define the \emph{Vershik map} $V_B : X^0_B \longrightarrow X_B$ that associates to each $x \in X^0_B$ its successor. The resulting pair $(X_B , V_B)$ is called \emph{Bratteli-Vershik dynamical system}.

\section{The Bratteli-Vershik process and stochastic machine}
\label{sec:BVP}

Here we will define the BV process but we need to introduce some new notation before it. 

Let $B = (V, E, \geq)$ be an ordered Bratteli diagram.

Recall the definition of $\zeta(x)$, for $x \in X^0_B$, from the previous section and define 
$$A(x)=\{1\leq i< \zeta(x):e_i(x) \textrm{ is not a minimal edge}\}.$$ 
Put $\theta(x) = \# A(x)$ and write $A(x)=\{k_{x,1},\ldots,k_{x,\theta(x)}\}$, where $k_{x,i-1}<k_{x,i}$, for all $i\in\{2,\ldots,\theta(x)\}$. 

Since for $k \in A(x)$ we have that  $e_k(x)$ is a maximal edge of $x$ which is not minimal which implies that $e_k(x)$ is not the only edge arriving at $r(e_k(x))$. Thus if $\# r^{-1}(v) > 1$ for every $v \in V \setminus  \{v_0\}$ or equivalently the sum of each row in each incidence matrix is greater than one, then we have that $\theta(x) = \zeta(x) - 1$ and $A(x)=\{1,\ldots,\zeta(x)-1\}$. So we have

\begin{center}
\textbf{Hypothesis A:} For the ordered Bratteli diagram $B = (V, E, \geq)$, the sum of each row in each incidence matrix is greater than one.
\end{center}

For each $j\in \{1,\ldots,\theta(x)\}$, let $y_{j}(x) \in X^0_B$ be defined as
\begin{equation}
y_{j}(x)=(f_1^{(j)},\ldots, f_{k_{x,j}}^{(j)},e_{k_{x,j}+1},e_{k_{x,j}+2},\ldots), \label{yjx}
\end{equation}
\noindent where $(f_1^{(j)},\ldots,f_{k_{x,j}}^{(j)})$ is the minimal edge in $E(1)\circ\ldots\circ E(k_{x,j})$ with range equal to $s(e_{k_{x,j}+1})$, for each $j\in\{1,\ldots, \theta(x)\}$.

First we need to adjust the space where the BV process will be defined. This is due to the fact that the successor of $x \in X^0_B$ can be an element of $X_B^{\max}$. To avoid this we define $\widehat{X}_B^{max}$ as the set of points $x \in X_B$ that are cofinal with a point on $X_B^{\max}$. Set
$$
\widehat{X}_B := X_B\setminus{\widehat{X}_B^{\max}} \, .
$$
Note that if $x \in \widehat{X}_B$ then $V_B(x) \in \widehat{X}_B$. Moreover $V_B$ restricted to $\widehat{X}_B$ is one to one from $\widehat{X}_B$ to $\widehat{X}_B \setminus X_B^{\min}$.

\begin{Definition} Let $(p_i)_{i\geq 1}$ be a sequence of nonnull probabilities and $B = (V, E, \geq)$ an ordered Bratteli diagram. The \emph{Bratteli-Vershik Process} is a discrete time-homogeneous Markov Process $(\Gamma_n)_{n\ge 0}$ with state space $\widehat{X}_B$ defined as
$$
\Gamma_n = \widehat{V}^{(n)}_B (\Gamma_0) \, ,
$$
where $\widehat{V}^{(n)}_B$ is the n-th iteration of $\widehat{V}_B : \widehat{X}_B \rightarrow \widehat{X}_B$ called the \emph{random Vershik map} and defined as
$$
\widehat{V}_B(x)=\left\{\begin{array}{cl}
y_{j}(x), & \textrm{ with probability } p_{k_{x,1}} \ldots p_{k_{x,j}}(1-p_{k_{x,j+1}}), \\
          &  \qquad \qquad \qquad \qquad \qquad \qquad \textrm{ for each } j\in \{1,\ldots,\theta(x)-1\}; \\
y_{\theta(x)}(x), & \textrm{ with probability } p_{k_{x,1}} \ldots p_{k_{x,\theta(x)}}(1-p_{\zeta(x)}), \\
x, & \textrm{ with probability } 1-p_{k_{x,1}}; \\
V_B(x), & \textrm{ with probability } p_{k_{x,1}} \ldots p_{k_{x,\theta(x)}} p_{\zeta(x)}.
\end{array}\right.
$$ \label{defvbs}
\end{Definition}

Thus the transition probabilities of the BV process are determined by the random Vershik map. The idea behind the definition is the use of a basic algorithm to obtain $V_B(x)$ from $x$ by recursively choosing the minimum path from level $0$ to level $k$ for $1 \le k \le \zeta(x)-1$ and then at step $\zeta(x)$ we finally obtain $V_B(x)$. Then we impose the rule that step $j$ of the algorithm is performed with probability $p_j$ independently of any other step. This transition mechanism is connected to the stochastic adding machines discussed in Section \ref{sec:intro} and our next aim is to define the BV stochastic adding machine.

\begin{remark} \label{remark:condA}
Under Hypothesis A we have that 
$$
\widehat{V}_B(x)=\left\{\begin{array}{cl}
y_{j}(x), & \textrm{ with probability } p_{1} \ldots p_{j}(1-p_{j+1}), \\
          & \qquad \qquad \qquad \qquad \qquad \qquad \textrm{ for each } j\in \{1,\ldots,\zeta(x)-1\}; \\
x, & \textrm{ with probability } 1-p_{1}; \\
V_B(x), & \textrm{ with probability } p_{1} \ldots p_{\zeta(x)-1} p_{\zeta(x)}.
\end{array}\right.
$$
\end{remark}

Take $x_0 \in \widehat{X}_B \cap X_B^{\min}$ and define $\widetilde{X}_B^{x_0} : = \{x_0\} \cup \{ V_B^{(n)}(x_0) : n \ge 1 \}$. Clearly we have a bijection between $\widetilde{X}_B^{x_0}$ and the set of nonnegative integers $\mathbb{Z}_+$ where $x_0 \mapsto 0$ and $V_B^{(n)}(x_{0})\mapsto n$ for all $n\geq 1$. Using the fact that $x_0 \in  X_B^{\min}$, it is also straightforward to verify that for every $x \in \widetilde{X}_B^{x_0}$ we have $\widehat{V}_B(x) \in \widetilde{X}_B^{x_0}$ with probability one.

To simplify the notation, we put $x_n := V_B^{(n)}(x_0)$ and then $\widetilde{X}_B^{x_0} = \{ x_0, x_1, x_2, ... \}$.

\begin{Definition} Let $(p_i)_{i\geq 1}$ be a sequence of nonnull probabilities, $B = (V, E, \geq)$ be an ordered Bratteli diagram and $x_0 \in \widehat{X}_B \cap X_B^{\min}$. The \emph{Bratteli-Vershik stochastic adding machine} associated to them is the discrete time-homogeneous Markov chain $(Y_n)_{n\ge 0}$ on $\widetilde{X}_B^{x_0}$ defined as $Y_n = \Gamma_n$ for $n\ge 1$ given that $Y_0 = x_0$. 
\end{Definition}

Let $(Y_n)_{n\ge 0}$ be a BV stochastic adding machine, we will denote the transition matrix of $(Y_n)_{n\ge 0}$ by $S=(S_{m,n})_{m,n\in\mathbb{N}}$, i.e 
\begin{equation} \label{defoperador}
S_{m,n} := S(x_n,x_m) := P(Y_1 = x_n |Y_0 = x_m).
\end{equation}

When $X_B^{\min} = \{ x_{\min}\}$ is a unit set, there is a unique BV stochastic adding machine associated to $B$ and a given sequence $(p_i)_{i\ge 1}$. This stochastic machine is the main object of study in this paper. To simplify notation we write $\widetilde{X}_B^{x_{\min}} = \widetilde{X}_B$. The hypothesis $X_B^{\min} = \{ x_{\min}\}$ is a natural one and occurs when the level sets $V_k$ are ordered and the order on the edges is endowed by the order on its source level sets. 

\begin{Example} (The Cantor systems of numeration case) \label{example:cantor} 
	
	Consider the ordered Bratteli diagram $B$ represented by the sequence of $1\times 1$ matrices $M_j=(d_j)$ for a sequence $d_j\ge 2$ for every $j\ge 1$. In this case we have a unique ordering which is the canonical consecutive ordering. Moreover Hypothesis A is clearly satisfied. In this case, $X^{min}_B$ is a unit set and given $(d_j)_{j\ge 1}$ and $(p_j)_{j\ge 1}$ there is a unique associated BV stochastic adding machine. The stochastic adding machines associated to the Cantor systems of numeration were introduced by Messaoudi and Valle \cite{mv}.
	
	For instance consider $d_j = 2j$, for all $j\geq 1$. Let $x=(e_1,e_2,e_3,e_4,\ldots)\in \widetilde{X}_B$, where $e_1=(1,1,1)$, $e_2=(1,3,1)$ and $e_3=(1,4,1)$. A representation of the path $(e_1,e_2,e_3)$ in the diagram is presented in item $(a)$ of Figure \ref{exemplodn}. Here we have $\zeta(x)=3$, because $e_1$ and $e_2$ are maximal edges and $e_3$ is not maximal. Thus $V_B(x)=(f_1,f_2,f_3,e_4,e_5,\ldots)$ where $f_1 = (1,0,1)$, $f_2=(1,0,1)$ and $f_3=(1,5,1)$. (see the item $b)$ of Figure \ref{exemplodn}). Moreover, we have $A(x)=\{1,2\}$ and $y_{1}(x)=(f_1,e_2,e_3,e_4,\ldots)$ and $y_{2}(x)=(f_1,f_2,e_3,e_4,\ldots)$ (see the items $(c)$ and $(d)$ of Figure \ref{exemplodn}, respectively). We have that $x$ transitions to $V_B(x)$ with probability $p_1p_2p_3$, $x$ transitions to $x$ with probability $1-p_1$, $x$ transitions to $y_{1}(x)$ with probability $p_1(1-p_2)$ and $x$ transitions to $y_{2}(x)$ with probability $p_1p_2(1-p_3)$. The initial parts of the transition graph and matrix for the chain are represented in Figure \ref{grafotrans2j}.
	
	\begin{figure}[!h]
		\centering
		\includegraphics[scale=0.46]{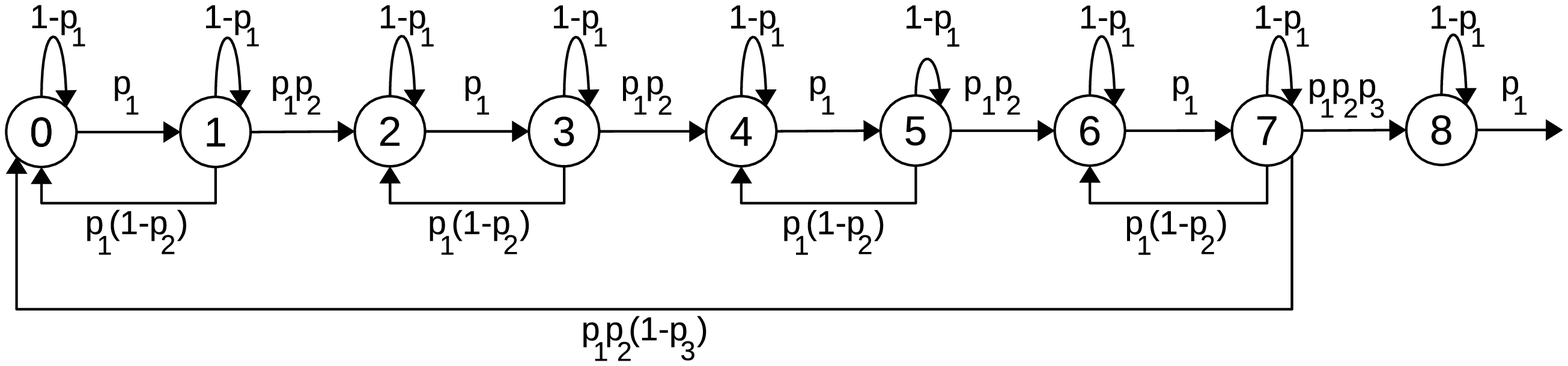}
		\includegraphics[scale=0.8]{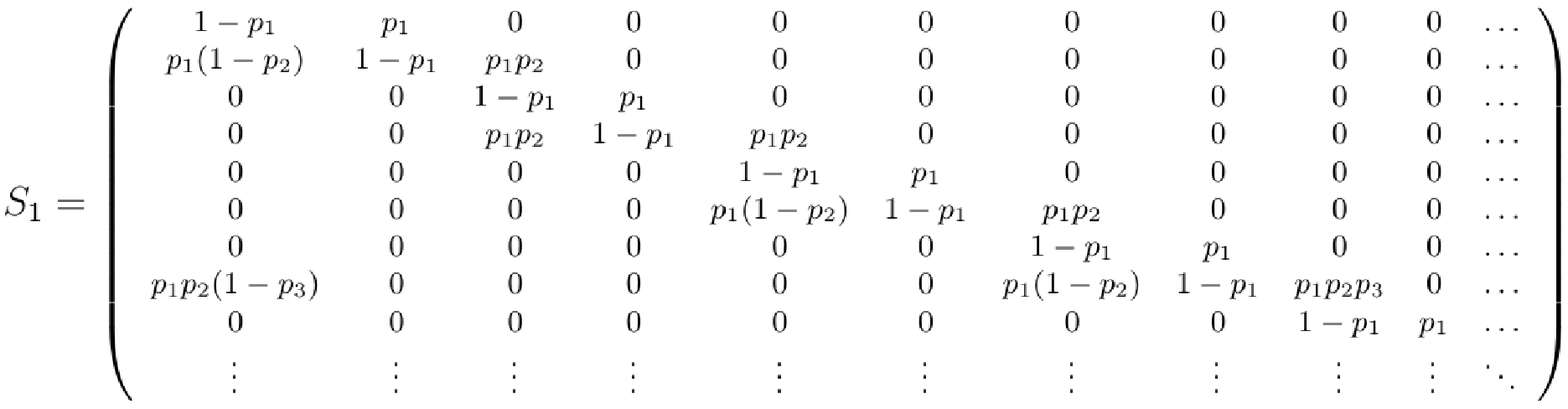}
		\caption{Initial parts of the transition graph and matrix of the BV stochastic adding machine with incidence matrices $M_j=(d_j)$ where $d_1 = 2$, $d_2 = 4$ and $d_3 = 6$.}
		\label{grafotrans2j}
	\end{figure}
	
\begin{figure}[!h]
		\centering
		\includegraphics[scale=0.8]{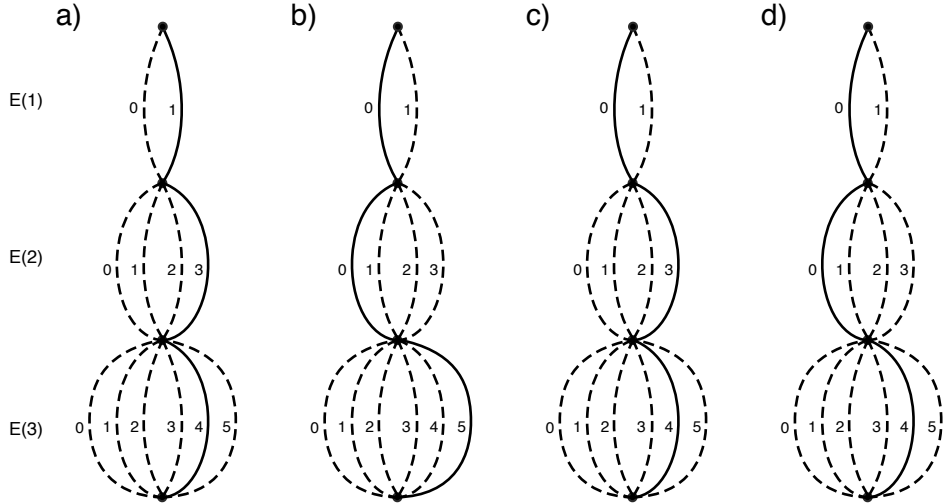}
		\caption{Representation paths in a Bratteli diagram with incidence matrices $M_j=(d_j)$ where $j\geq 1$, $d_1 = 2$, $d_2 = 4$ and $d_3 = 6$.}
		\label{exemplodn}
	\end{figure}
\end{Example}

\begin{remark}
In Example \ref{example:cantor}, if $d_j=2$ for all $j\geq 1$, then we obtain the stochastic adding machine defined by Killeen and Taylor \cite{kt1}.
\end{remark}

\begin{Example} \label{exemplo24}
	Consider $B$ as the stationary Bratteli diagram with consecutive ordering and incidence matrix $M_1=\left(\begin{array}{cc} 2 & 1 \\ 3 & 1 \end{array}\right)$. This diagram satisfies Hypothesis A. 
	
	Let $x=(e_1,e_2,e_3,e_4,e_5,\ldots)\in X_B$ be an infinite path, where $e_1=(2,3,2)$, $e_2=(2,2,1)$, $e_3=(1,1,2)$, $e_4=(2,2,1)$ and $e_j = (1,0,1)$ for $j\ge 5$. The representation of $x$ in the diagram is given by the path in item $(a)$ of Figure \ref{exemplo}.
	
	Here we have $\zeta(x)=3$ and $V_B(x)=(f_1,f_2,f_3,e_4,e_5,\ldots)$ where $f_1 = (1,0,1)$, $f_2 = (1,0,1)$ and $f_3=(1,2,2)$. (see item $(b)$ of Figure \ref{exemplo}).
	
	Moreover, we have $A(x)=\{1,2\}$ and $y_{1}(x)=((1,0,2),e_2,e_3,\ldots)$ and $y_{2}(x)=((1,0,1),(1,0,1),e_3,e_4,\ldots)$ (see the $(c)$ and $(d)$ of Figure \ref{exemplo}, respectively).
	
	Hence, we have that $x$ is mapped to $V_B(x)$ with probability $p_1p_2p_3$, $x$ is mapped to $x$ with probability $1-p_1$, $x$ is mapped to $y_{1}(x)$ with probability $p_1(1-p_2)$ and $x$ is mapped to $y_{2}(x)$ with probability $p_1p_2(1-p_3)$.
	
	Thus, its transition graph and transition operator are represented in Figure \ref{grafotrans2131}.
	
	\begin{figure}[!h]
		\centering
		\includegraphics[scale=0.46]{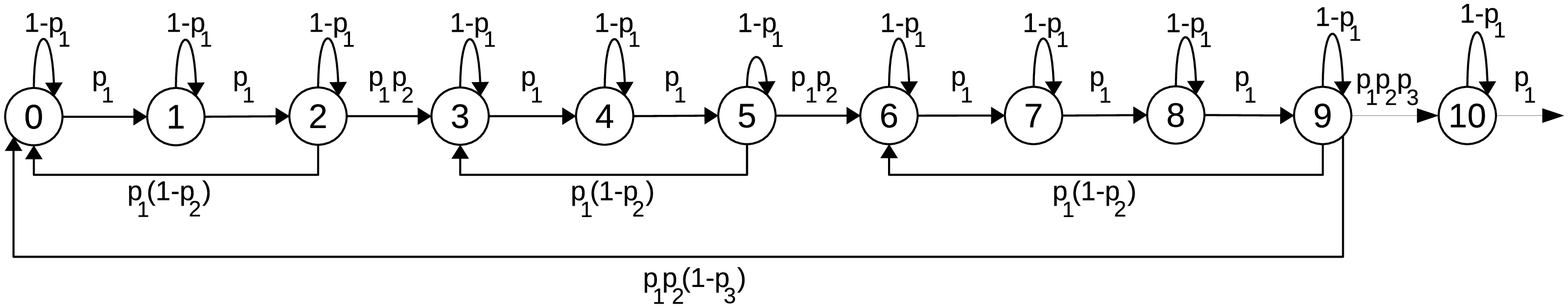}
		\includegraphics[scale=0.8]{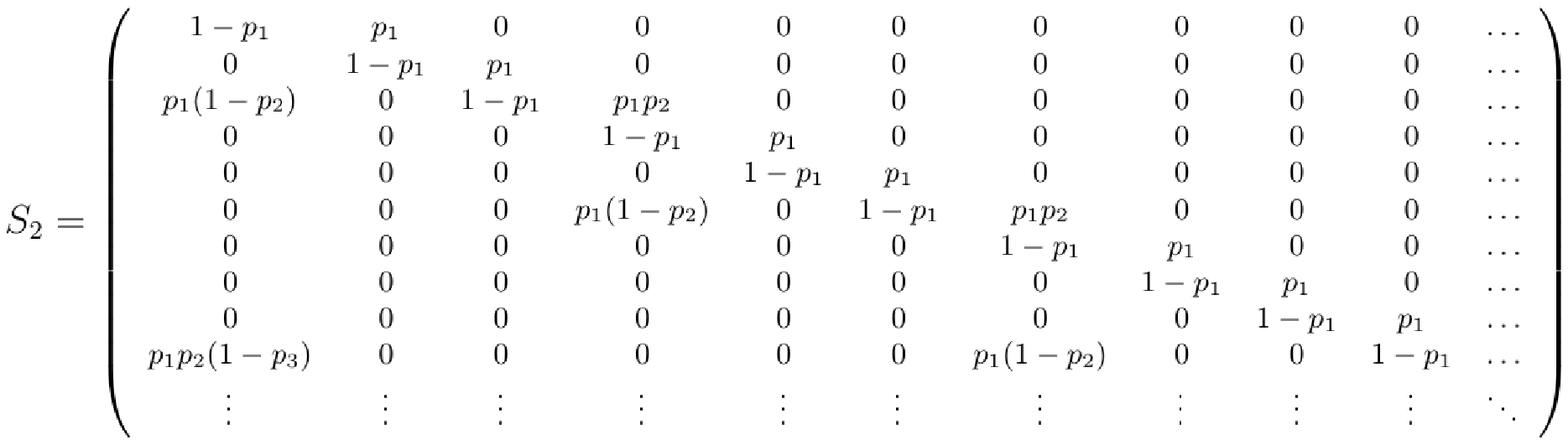}
\caption{Initial parts of the transition graph and matrix of the BV stochastic adding machine associated with a stationary Bratteli diagram with incidence matrix $M_1$.}
		\label{grafotrans2131}
	\end{figure}
	
\begin{figure}[h!]
		\centering
		\includegraphics[scale=0.8]{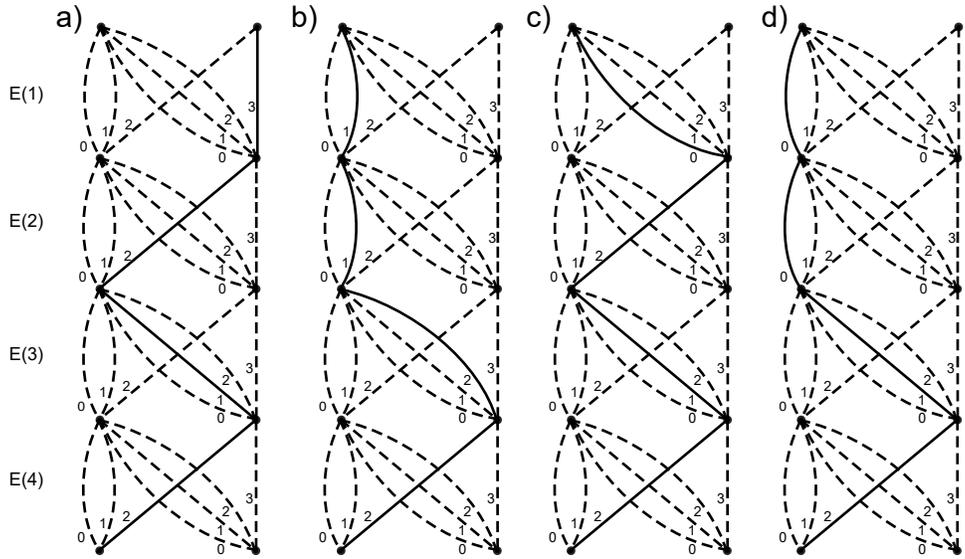}
\caption{Representation of paths in a stationary Bratteli diagram with incidence matrix $M_1$.}\label{exemplo}
	\end{figure}
\end{Example}

\begin{Example} \label{example:fibo}
	(The Fibonacci case) 
	
	Consider the stationary ordered Bratteli diagram $B$ with the canonical consecutive ordering and incidence matrix $M_F=\left(\begin{array}{cc} 1 & 1 \\ 1 & 0 \end{array}\right)$. In this case $B$ does not satisfy Hypothesis A. Again $X^{min}_B$ is unitary and given $(p_j)_{j\ge 1}$ there is a unique associated BV stochastic adding machine. These stochastic adding machines is associated with the Fibonacci system of numeration and have been introduced in \cite{ms}
	
	Let $x=(e_1,e_2,e_3,e_4,\ldots)\in X_B$ be an infinite path in the Bratteli diagram, where $e_1=(2,1,1)$, $e_2=(1,0,2)$, $e_3=(2,1,1)$, and $e_j = (1,0,1)$ for all $j\ge 4$. The representation of $x$ in the diagram is given by the continuous path in item $(a)$ of Figure \ref{diagramfibonacci}. We have $\zeta(x)=4$ and $V_B(x)=(f_1,f_2,f_3,f_4,e_5,\ldots)$ where $f_4=(2,1,1)$ and $(f_1,f_2,f_3)$ is the minimal edge in $E(1)\circ E(2)\circ E(3)$ with range equal to $s(f_4)$. (see the item $(b)$ of Figure \ref{diagramfibonacci}). 
	
	We have $A(x)=\{1,3\}=\{n_1,n_2\}$ and $y_{n_1}(x)=((1,0,1),(1,0,2),(2,1,1),e_4,\ldots)$ and $y_{n_2}(x)=((1,0,1),(1,0,1),(1,0,1),e_4,\ldots)$ (see the items $(c)$ and $(d)$ of Figure \ref{diagramfibonacci}, respectively).
	
	Hence, we have that $x$ transitions to $V_B(x)$ with probability $p_1p_2p_3$, $x$ transitions to $x$ with probability $1-p_1$, $x$ transitions to $y_{n_1}(x)$ with probability $p_1(1-p_2)$ and $x$ transitions to $y_{n_2}(x)$ with probability $p_1p_2(1-p_3)$.
	
	\begin{figure}[!h]
		\centering
		\includegraphics[scale=0.8]{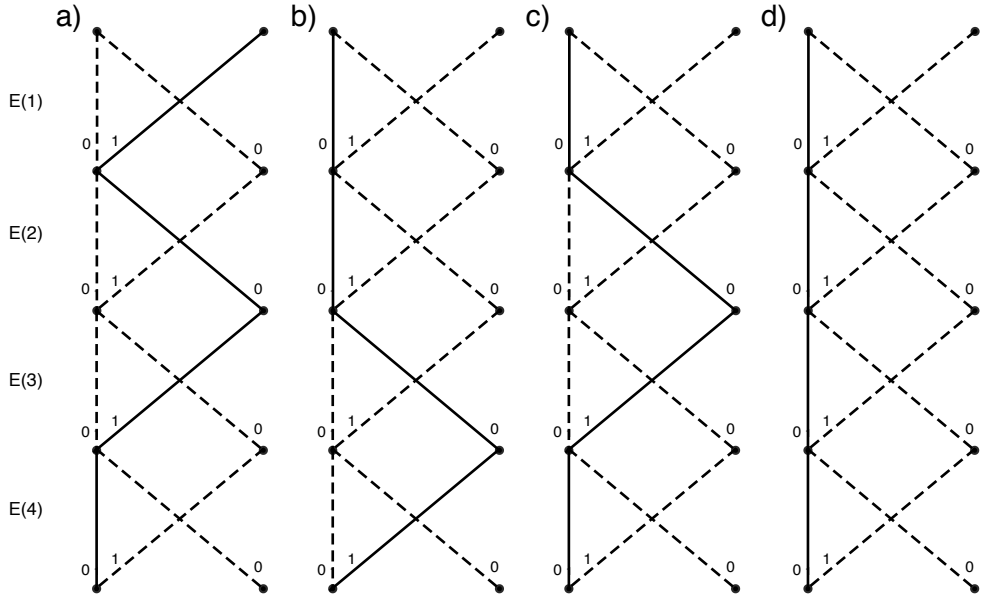}
\caption{Representation of paths in a stationary Bratteli diagram with incidence matrix $M_F$.}
		\label{diagramfibonacci}
	\end{figure}
	\end{Example}

\begin{remark} \label{remark:uniq}
Two distinct ordered Bratteli diagrams can generate the same stochastic adding machine. For instance consider two stationary ordered Bratteli diagrams with consecutive ordering and incidence matrices $M=(2)$ and $M^\prime =\left(\begin{array}{cc} 1 & 1 \\ 1 & 1 \end{array}\right)$. Both diagrams generate a unique BV stochastic adding machine that corresponds to the stochastic machine studied by Killeen and Taylor in \cite{kt1}.
\end{remark}

Before we discuss the probabilistic properties of the BV stochastic adding machines, we present some basic definitions from the theory of Markov chains and we recommend \cite{bre} to the unfamiliar reader. Let $Y = (Y_n)_{n\geq 0}$ be a Markov Chain on a probability space $(\Omega, \mathcal{O},P)$. We denote by $\mathbb{E}[\cdot]$ the expectation with respect to $P$. We say that $Y$ is irreducible if 
for any pair of states $i$ and $j$ there exists $m \ge 1$ such that
$$
P(Y_m = j | Y_0 = i ) \, >0.
$$
An irreducible Markov chain $Y$ is transient if every state $i$ is transient, i.e. 
\begin{center}
$P(Y_n=i \textrm{ for some } n |Y_0=i) < 1$,
\end{center}
If an irreducible Markov chain is not transient we say that it is recurrent and this means that every state $i$ is recurrent, i.e.
$$
P(Y_n=i \textrm{ for some } n |Y_0=i) = 1.
$$
Furthermore, a recurrent Markov chain is called positive recurrent if for each state $i$, the expected return time $m_i = \mathbb{E}[R_i|Y_0=i] < \infty$, where $R_i=\min\{n\geq 1: Y_n=i\}$.
Otherwise, if $m_i=+\infty$, then the Markov chain is called null recurrent.

\begin{Proposition}
\label{prop:irredutivel}
Let $(p_i)_{i\geq 1}$ be a sequence of nonnull probabilities such that $\# \{ i: p_i < 1 \} = \infty$. Every BV stochastic adding machine associated to $(p_i)_{i\geq 1}$ is an irreducible Markov chain. Furthermore the stochastic machine is transient if and only if $\prod_{j=1}^{\infty} p_j > 0$. 
\end{Proposition}
\begin{proof}
Let $(Y_n)_{n\ge 0}$ be a BV stochastic adding machine associated to $(p_i)_{i\geq 1}$, an ordered Bratteli diagram $B = (V, E, \geq)$ and $x_0 \in \widehat{X}_B \cap X_B^{\min}$. 

We have some special states $x_{n_1}$, $x_{n_2}$, ..., which are cofinal to $x_0$ by hypothesis, determined by the following:  
$e_k (x_{n_j}) = e_k (x_0)$ for $k\ge j+1$ and $(e_1 (x_{n_j}),...,e_j (x_{n_j}))$ is the maximal edge in $E(1) \circ ... \circ E(j)$ with range equal so $s(e_{j+1} (x_0))$.

Concerning irreducibility, we just point out that 
\begin{enumerate} 
\item[(i)] for every $n$ the chain can reach $x_n$ with positive probability by making the transitions $x_0 \mapsto x_1$, $x_0 \mapsto x_1$, ..., $x_{n-1} \mapsto x_n$;
\item[(ii)] for $j+1 \in \{ i: p_i < 1 \}$, we can make the transition $x_{n_j} \mapsto x_0$ with probability $(1-p_{j+1}) \prod_{i=1}^j p_j > 0$.
\end{enumerate}
By (i) and (ii), it is clear that $(Y_n)_{n\ge 0}$ is irreducible. 

\smallskip

Now we consider the transience/recurrence of the chain. We rely on some additional properties of the chain related to the special states $x_{n_j}$, $j\ge 1$. We have
\begin{enumerate} 
\item[(iii)] Once the chain arrives at $x_{n_j +1}$, the successor of $x_{n_j}$, it can only visit $x_{n_j}$ again if it visits $x_0$ first. 
\item[(iv)] If transition $x \mapsto x_0$ is possible with positive probability, then $x = x_{n_j}$. 
\item[(v)] Given that a transition from $x_{n_j}$ to $x_{n_j +1}$ or $x_0$ occurs, the next state of the chain is $x_{n_j +1}$ with probability $p_{j+1}$, i.e
$$
P\big( Y_{n+1} = x_{n_j +1} \big| Y_n = x_{n_j}, Y_{n+1} \in \{ x_0 , x_{n_j +1} \} \big) = p_{j+1} \, .
$$
\end{enumerate}
The verification of (iii), (iv), (v) follows directly from the definition of $(Y_n)_{n\ge 0}$. By the Markov property and properties (i)-(v) above, the probability that the $(Y_n)_{n\ge 0}$ never returns to $x_0$ coincide with the event that $(Y_n)_{n\ge 0}$ reach $x_{n_j}$ before it returns to $x_0$ for every $j\ge 1$ which has probability $\prod_{j=1}^{\infty} p_j > 0$.
\end{proof}

\begin{remark}
Let $(Y_n)_{n\ge 0}$ be an irreducible BV stochastic adding machine. If $p_1 <1$ then clearly $(Y_n)_{n\ge 0}$ is aperiodic since $P(Y_1 = x_0 |Y_0 = x_0) = 1-p_1 > 0$. However, when $p_1 = 1$ the chain can be periodic or aperiodic depending on the Bratteli diagram.
\end{remark}

\begin{Proposition}
\label{prop:rec-}
Let B be an ordered Bratteli diagram satisfying Hypothesis A and $(p_i)_{i\geq 1}$ be a sequence of nonnull probabilities such that $\# \{ i: p_i < 1 \} = \infty$ and $\prod_{j=1}^{\infty} p_j = 0$. Then every BV stochastic adding machine associated to $(p_i)_{i\geq 1}$ is null recurrent. 
\end{Proposition}
\begin{proof}
Let $(Y_n)_{n\ge 0}$ be a BV stochastic adding machine associated to $(p_i)_{i\geq 1}$, an ordered Bratteli diagram $B = (V, E, \geq)$ and $x_0 \in \widehat{X}_B \cap X_B^{\min}$. Suppose that $B = (V, E, \geq)$ satisfies Hypothesis A, $\# \{ i: p_i < 1 \} = \infty$ and $\prod_{j=1}^{\infty} p_j = 0$. By Proposition \ref{prop:irredutivel}, the chain is irreducible and recurrent.

Put $T = \inf\{ n \ge 1 : Y_n = x_0 \}$, i.e the first return time to $x_0$. We are going to show that the expected value of $T$, $\mathbb{E}[T]$, is infinite and then the chain is null recurrent.

To compute $\mathbb{E}[T]$ we need to recall the definition of the special states $x_{n_j}$, $j\ge 1$, and their properties from the proof of Proposition \ref{prop:irredutivel}. Also recall the definition of the transition probabilities under Hypothesis A from Remark \ref{remark:condA}.

Put $x_{n_0}:=x_0$ and consider the following decomposition
$$
T = \sum_{n=0}^\infty T I_{\{Y_{T-1} = x_{n_j}\}} \, ,
$$
where $I_W$ is the indicator function of the event $W$. We obtain that 
\begin{equation}
\label{eq:ET}
\mathbb{E}[T] = \sum_{n=0}^\infty \mathbb{E}[T | I_{\{Y_{T-1} = x_{n_j}\}}] P(Y_{T-1} = x_{n_j}) \, .
\end{equation}
Clearly on $\{Y_{T-1} = x_{n_0}\}$ we have $T=1$ and $P(Y_{T-1} = x_{n_0}) = 1 - p_1$. Using item (v) in the proof of Proposition \ref{prop:irredutivel} we get that
\begin{equation}
\label{eq:PT}
P(Y_{T-1} = x_{n_j}) = \Big( \prod_{i=1}^j p_j \Big) (1- p_{j+1}) \, .
\end{equation}
We also have that 
\begin{equation}
\label{eq:ET1}
\mathbb{E}[T | I_{\{Y_{T-1} = x_{n_0}\}}] = 1 \, .
\end{equation}

\textbf{Claim:} For every $j\ge 1$
$$
\mathbb{E}[T | I_{\{Y_{T-1} = x_{n_j}\}}] \ge 1 + \sum_{i=1}^j \Big( \prod_{r=1}^i p_r \Big)^{-1}
$$ 

Suppose that the claim holds. Then by \eqref{eq:ET} and \eqref{eq:PT} we have that
\begin{eqnarray}
\mathbb{E}[T] & \ge & (1-p_1) \\
& & + \Big( 1 + \frac{1}{p_1} \Big) p_1 (1-p_2) \\
& & + \Big( 1 + \frac{1}{p_1} + \frac{1}{p_1 p_2} \Big) p_1 p_2 (1-p_3) + \cdots 
\end{eqnarray}
Rearranging terms and putting $p_0=1$ we obtain 
\begin{eqnarray}
\mathbb{E}[T] & \ge & \sum_{m=0}^\infty \sum_{j=1}^\infty p_{m} .... p_{m+j-1} (1-p_{m+j})   \\
& = & \sum_{m=0}^\infty (1- \prod_{j\ge m+1} p_j) = \sum_{m=0}^\infty 1 = \infty \, . 
\end{eqnarray}
Thus the chain is null recurrent.

It remains to prove the Claim. We prove it by induction. Suppose the claim holds for $j-1$ (the case $j=0$ is \eqref{eq:ET1}). Given $\{Y_{T-1} = x_{n_j}\}$ write $T= T_1 + T_2$ where $T_1$ is the time of the first visit of the chain to $x_{n_{j-1}+1}$ and $T_2$ the time spent on $\{x_{n_{j-1}+1},...,x_{n_{j}}\}$ until it arrives at $x_0$. By the induction hypothesis
$$
\mathbb{E}[T_1 | I_{\{Y_{T-1} = x_{n_j}\}}] \ge 1 + \sum_{i=1}^{j-1} \Big( \prod_{r=1}^i p_r \Big)^{-1} .
$$ 
It remains to prove that
$$
\mathbb{E}[T_2 | I_{\{Y_{T-1} = x_{n_j}\}}] \ge \Big( \prod_{r=1}^j p_r \Big)^{-1} .
$$
Time $T_2$ is greater or equal to the number of transitions to get to $x_{0}$ from $x_{n_{j}}$, and this is bounded below by the necessary number of trials from $j$ independent Bernoulli random variables with parameters $p_1$, ... ,$p_j$ to obtain $j$ successes. It is an exercise in probability theory using geometric random variables to prove that this number of trials have expected value equal to $\Big( \prod_{r=1}^j p_r \Big)^{-1}$. 
\end{proof}

From the proof of Proposition \ref{prop:rec-} we see that we can drop Hypothesis $A$ if the sequence $(p_i)_{i\geq 1}$ is constant and the Bratteli diagram is stationary.

\begin{Proposition} \label{prop:rec-p}
Let B be a stationary ordered Bratteli diagram. If $p_i=p \in (0,1)$ for every $i\ge 1$, then every BV stochastic adding machine associated to $(p_i)_{i\geq 1}$ is null recurrent.
\end{Proposition}

Although we have Propositions \ref{prop:rec-} and \ref{prop:rec-p}, a BV stochastic adding machine associated to $(p_i)_{i\ge 1}$ such that $\prod_{j=1}^{\infty} p_j = 0$ can be positive recurrent. So Hypothesis A is necessary. In Example \ref{example:fibo} we describe a stationary BV stochastic adding machine associated to an ordered Bratteli diagram with consecutive ordering which can be positive recurrent for a sufficiently fast decreasing sequence $(p_i)_{i\ge 1}$.

%%%%%%%%%%%%%%%%%%%%%%%%%%%%%%%%%%%%%%%%%%%%%%%%%%%%%%%%%%%%%%%%%%%%%%%%%%%%%%%%%%%%%%%%%%%
%%%%%%%%%%%%%%%%%%%%%%%%%%%%%%%%%%%%%%%%%%%%%%%%%%%%%%%%%%%%%%%%%%%%%%%%%%%%%%%%%%%%%%%%%%%

\section{Stochastic machines of stationary $2\times 2$ Bratteli diagrams}
\label{sec:spectrum}

Let $B=(V,E,\geq)$ be a stationary simple ordered Bratteli diagram with incidence matrix $M=\left(\begin{array}{cc} a & b \\ c & d \end{array}\right)$. 

Since $B$ is simple, we have necessarily $b>0$ and $c>0$, moreover either $a>0$ or $d>0$. We can change the labels of vertices in $B$ if necessary and suppose that $a>0$. Therefore $a+b > 1$ and Hypothesis A is equivalent to $c+d > 1$.

\smallskip

We start with a proposition that gives a condition on $2\times 2$ Bratteli diagrams that allows the existence of positive recurrent BV stochastic adding machines.

\begin{Proposition} \label{prop:recpos2x2}
Let $B=(V,E,\geq)$ be a $2\times 2$ stationary simple ordered Bratteli diagram with $a=c=1$, $b>0$ and $d=0$. Then the BV stochastic adding machine associated to $(p_j)_{j\ge 1}$ is positive recurrent if $p_j$ decreases to zero sufficiently fast as $j\rightarrow \infty$.
\end{Proposition}
\begin{proof}
Recall the definitions from the proof of Proposition \ref{prop:rec-}. In order to prove that the stochastic machine is positive recurrent we have to show that $\mathbb{E}[T] < \infty$.

We claim that there exists $(C_j)_{j\ge 1}$ depending on $b$ but not on $(p_j)_{j\ge 1}$ such that
\begin{equation}
\label{eq:ET<}
\mathbb{E}[T] \le C_1 + \sum_{j=1}^{+\infty} C_j \max \{ p_{j-1} , p_j \} \, .
\end{equation}
From the previous inequality, one simply need to choose $p_j \le r_j / (C_j + C_{j+1})$ with $\sum_{j=1}^{+\infty} r_j < \infty$.

To prove \eqref{eq:ET<} we use \eqref{eq:ET} and \eqref{eq:PT}. So we need to bound from above the conditional expectation $\mathbb{E}[T | I_{\{Y_{T-1} = x_{n_j}\}}]$. The particular form of $x_{n_j}$ is important here. We have that
$$
x_{n_1} = \big( (2,b,1) , (1,0,1) , (1,0,1) , ... \big) \, ,
$$
thus the time to get to $x_{n_1} + 1$ from $x_0$ given $Y_{T-1} = x_{n_j}$ is equal to one plus a negative binomial distribution with parameters $b$ and $p_1$ because the chain uses one unit of time to leave $x_0$ and then spend a geometric time of parameter $p_1$ on each of the last $b$ edges of $E(1)$ with range $1\in V(1)$. Therefore
$$
\mathbb{E}[T | I_{\{Y_{T-1} = x_{n_1}\}}] = 1 + \frac{b}{p_1} 
$$ 
and 
$$
\mathbb{E}[T | I_{\{Y_{T-1} = x_{n_1}\}}] P(Y_{T-1} = x_{n_1}) \le 1+ b = C_1 \, .
$$
Before we can use induction on $j$ we still need to deal with $\mathbb{E}[T | I_{\{Y_{T-1} = x_{n_2}\}}]$ and we need to compute the mean time to get to $x_{n_2} + 1$ from $x_{n_1} + 1$. We have
$$
x_{n_1} + 1 = \big( (1,0,2) , (2,1,1) , (1,0,1) , (1,0,1) ... \big) \, ,
$$
where the first edge is the unique edge in $E(1)$ with range $2$. So from $x_{n_1} + 1$ we only need to change $b$ edges in $E(2)$ to get to $x_{n_2} + 1$ and on each of these edges we spend a geometric time of parameter $p_2$. Therefore 
$$
\mathbb{E}[T | I_{\{Y_{T-1} = x_{n_2}\}}] = \mathbb{E}[T | I_{\{Y_{T-1} = x_{n_1}\}}] + \frac{b}{p_2} =
1 + \frac{b}{p_1} + \frac{b}{p_2} \, ,
$$
and $\mathbb{E}[T | I_{\{Y_{T-1} = x_{n_2}\}}] P(Y_{T-1} = x_{n_2})$ is bounded above by
$$
p_1 p_2 + b p_2 + b p_1 \le (1+2b) \max \{ p_1 , p_2\} 
= C_2 \max \{ p_1 , p_2\}  \, .
$$
Analogous estimates allow us to show that $\mathbb{E}[T | I_{\{Y_{T-1} = x_{n_3}\}}] P(Y_{T-1} = x_{n_3})$ is bounded above by
$$
p_1 p_2 p_3 + p_2 p_3 b + p_1 p_3 b + p_2 b^2 \le (1+2b+b^2) \max \{ p_2 , p_3\} \, .  
$$

Now Suppose that 
$$
\mathbb{E}[T | I_{\{Y_{T-1} = x_{n_j}\}}] P(Y_{T-1} = x_{n_j}) \le C_j \max \{ p_{j-1} , p_j \} \, ,
$$ 
and we are going to estimate $\mathbb{E}[T | I_{\{Y_{T-1} = x_{n_{j+2}}\}}]$. Using the fact that $a=c=1$ to go from $x_{n_{j+1}} + 1$ to $x_{n_{j+2}} + 1$ we need to change $b$ edges in $E(j+2)$ without change the edge $(j+1,1,0,2) \in E(j+1)$ but considering all edges in $E(1) \circ ... \circ E(j)$ with range $1 \in V(j)$. Thus
$$
\mathbb{E}[T | I_{\{Y_{T-1} = x_{n_{j+2}}\}}] \le \mathbb{E}[T | I_{\{Y_{T-1} = x_{n_{j}}\}}] + ( \mathbb{E}[T | I_{\{Y_{T-1} = x_{n_{j+2}}\}}] - 1 ) \frac{b}{p_{j+2}} \, .
$$
Thus $\mathbb{E}[T | I_{\{Y_{T-1} = x_{n_{j+2}}\}}] P(Y_{T-1} = x_{n_{j+2}})$ is bounded above by
$$
C_j \max \{ p_{j-1} , p_j \} p_{j+2} + C_j \max \{ p_{j-1} , p_j \} p_{j+1} \le 2 C_j \max \{ p_{j+1} , p_{j+2} \}  \, .
$$
So we just need to take $C_{j+2} = 2 C_j$. 
\end{proof}

From the proof of Proposition \ref{prop:recpos2x2} we can also see that it is enough to have $2^{-j} p_j$ summable to obtain a positive recurrent stochastic machine from the hypothesis of the proposition.

\medskip

As a corollary we get the result from \cite{cap} about the existence of positive recurrent Fibonacci stochastic adding machines.

\begin{Corollary}
The Fibonacci stochastic adding machines associated to $(p_j)_{j\ge 1}$ are positive recurrent if $p_j$ decreases to zero sufficiently fast as $j\rightarrow \infty$.
\end{Corollary}

To continue the study of BV stochastic machines of $2\times 2$ Bratteli diagrams, we need to introduce some notation related to systems of numeration associated to the $2\times 2$ Bratteli diagrams.

Let us denote $M^n$ by 
$$
M^n =\left(\begin{array}{cccc}
a_n & b_n \\
c_n & d_n
\end{array}\right),
$$
for all $n\geq 0$, where $M^0=I$ is the identity matrix. For each $n\geq 0$, put
$$
F_n=a_n+b_n \, , \ \ G_n=c_n+d_n  \ \ \forall n\ge 1 \, .
$$
This gives $F_0=G_0=1$, $F_1 = a+b$ and $G_1 = c+d$.

\begin{Remark} \label{caminhos}
For each nonnegative integer $n\geq 1$, $F_n$ is the number of paths from $V(0)$ to the vertex $(n,1)$ at the Bratteli diagram. Respectively, $G_n$ is the number of paths from $V(0)$ to the vertex $(n,2)$.
\end{Remark}

\begin{Lemma} \label{lemmaseq}
We have $F_{n+1}=(a+d)F_n-(ad-bc)F_{n-1}$ and $G_{n+1}=(a+d)G_n-(ad-bc)G_{n-1}$, for all $n\geq 1$.\label{case2}
\end{Lemma}
\begin{proof}
It comes from the fact that $\left(\begin{array}{c}
F_n \\ G_n
\end{array}\right) = M^n \left(\begin{array}{c}
1 \\ 1 
\end{array}\right)$ for all $n\geq 0$ and the characteristic polynomial of $M$ is $p(x)=x^2+(a+d)x-(ad-bc)$.
\end{proof}

\subsection{$2\times 2$ case under Hypothesis A and consecutive ordering}

From now on we assume that $abc>0$, $c+d>1$ and that $B$ is endowed with the consecutive ordering. Thus $B$ is simple and satisfies Hypothesis A. Moreover $X^{min}_B = \{x_0\}$ is a unitary set and for each $x\in\widetilde{X}_B$ we have $A(x)=\{1,\ldots,\zeta(x)-1\}$. The aim of this section is the study of the spectrum of BV stochastic machines under these conditions.

We first need to establish a proper notation to deal with the possible transitions of the chain in $\widetilde{X}_B = \widetilde{X}_B^{x_0}$. Define $\vec{0}_j$ as the minimum edge of $E(j)$ with range $1$, i.e.
$\vec{0}_j = (j,1,0,1)$. For convenience we will sometimes not write the level index $j$ simply writing $\vec{0}=(1,0,1)$.  Let $x = (e_j)_{j\ge 1} =((s_j,m_j,r_j))_{j\ge 1} \in \widetilde{X}_B$. Recall that $x_0 = (\vec{0}_j)_{j\ge 1}$ and $x \neq x_0$ is cofinal with $x_0$, thus there exists $N \in\mathbb{N}$ such that $x_N = V_B^N(x_0)=x$. Put
$$
\xi(x) = \min \{ j \ge 1: e_l = \vec{0} \ \textrm{ for all } \ l> j \} \, .
$$ 
The reader should recall the definition of $\zeta(x)$ and note that $\zeta(x)$ and $\xi(x)$ play a different role.

\subsection{Numeration systems associated to Bratteli diagrams}

\begin{Definition}\label{defdeltagamma}
Let $x \in \widetilde{X}_B$. For each $j\in\{1,\ldots,\xi(x)\}$, define $\delta_j(x)=\delta_j$ and $\gamma_j(x)=\gamma_j$ according to the following four cases:
\begin{enumerate}
\item[(i)] If $s_j = 1$ and $r_j =1$  then $\delta_j = m_j\in\{0,\ldots,a-1\}$ and $\gamma_j = 0$;
\item[(ii)] If $s_j = 2$ and $r_j = 1$ then $\delta_j = a$ and $\gamma_j = m_j - a \in\{0,\ldots, b-1\}$;
\item[(iii)] If $s_j = 1$ and $r_j = 2$ then $\delta_j = m_j\in\{0,\ldots,c-1\}$ and $\gamma_j = 0$;
\item[(iv)] If $s_j=2$ and $r_j = 2$ then $\delta_j = c$ and $\gamma_j = m_j - c \in\{0,\ldots,d-1\}$.
\end{enumerate}
For $x_N = V_B^N(x_0)=x$ we also denote $\delta_j = \delta_j(N)$ and $\gamma_j = \gamma_j(N)$.
\end{Definition}

Observe that $m_j=\delta_j+\gamma_j$, for all $j\geq 1$. Moreover, if $d=0$, then $(s_j,r_j)\neq (2,2)$, for all $j\geq 1$.

\begin{Example} \label{exdeltagamma}
	Consider the consecutive ordering Bratteli diagram $B$ represented by the matrix $M=\left(\begin{array}{cc} 1 & 3 \\ 1 & 4 \end{array}\right)$. By Lemma \ref{lemmaseq}, we have
	\begin{center}$ \begin{array}{llllll}
	F_0=1, & F_1=4, & F_2=19, & F_3=91, &  \ldots \\
	G_0=1, & G_1=5, & G_2=24, & G_3=115,  & \ldots 
	\end{array}$ \end{center} 
	
Consider $x,y\in B$ where $x=(x_j)_{j\geq 1} = ((2,3,2),(2,4,2),(2,2,1),\vec{0},\vec{0},...)$ and $y=(y_j)_{j\geq 1}=((2,2,1),(1,0,2),(2,3,1),\vec{0},\vec{0},...)$. The representation of $x$ and $y$ in the Bratteli diagram is given respectively in the items $(a)$ and $(b)$ of Figure \ref{digdigbrat}.
	\begin{figure}[h!]
		\centering
		\includegraphics[scale=0.8]{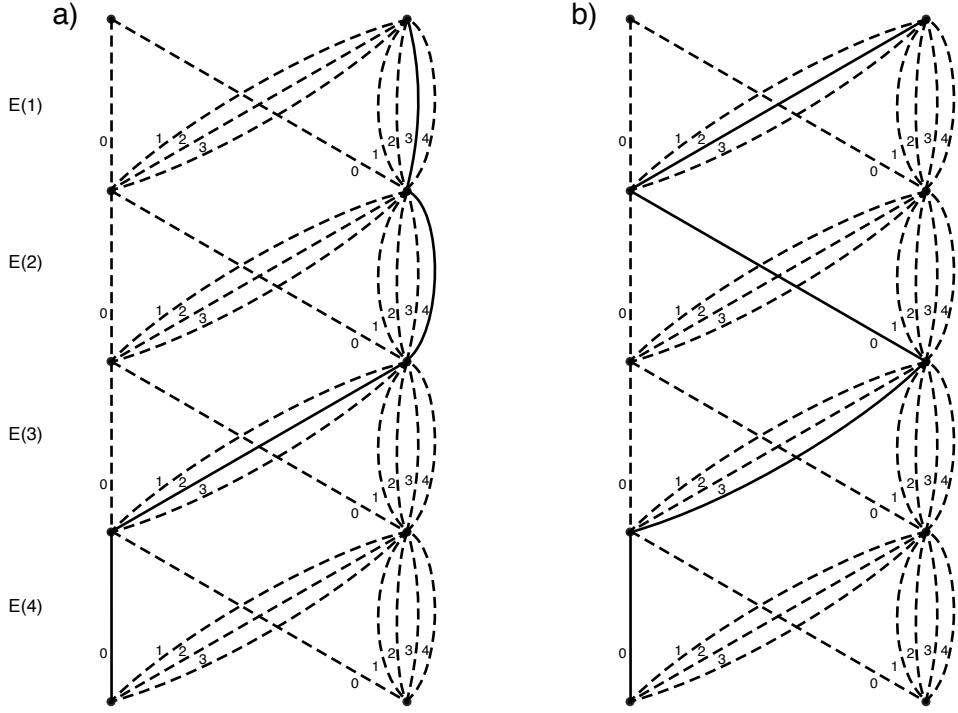}
		\caption{Representation of the paths in the Bratteli diagram described in Example \ref{exdeltagamma}.}
		\label{digdigbrat}
	\end{figure}
	
	By Definition \ref{defdeltagamma}, we have that
\begin{center}	$\begin{array}{l}
	\delta_1(x)=1, \textrm{ } \gamma_1(x)=2;   \\
	\delta_2(x)=1, \textrm{ } \gamma_2(x)=3;   \\
	\delta_3(x)=1, \textrm{ } \gamma_3(x)=1;   \\
	\delta_i(x)=\gamma_i(x)=0 \textrm{, for all } i\geq 4.
	\end{array} \textrm{ and } \begin{array}{l}
	\delta_1(y)=1, \textrm{ } \gamma_1(y)=1;   \\
	\delta_2(y)=0, \textrm{ } \gamma_2(y)=0;   \\
	\delta_3(y)=1, \textrm{ } \gamma_3(y)=2;   \\
	\delta_i(y)=\gamma_i(y)=0 \textrm{, for all } i\geq 4.
	\end{array}$ \end{center} 
\end{Example}
\begin{Proposition} \label{representation}
Let $N$ be a nonnegative integer and $x \in \widetilde{X}_B$ such that $V_B^N(x_0) = x$. Then, $N=\sum_{j=0}^{\xi(x)}\delta_{j+1} F_j+\gamma_{j+1} G_j$, where $\delta_i(N)=\delta_i$ and $\gamma_i(N)=\gamma_i$ are defined in Definition \ref{defdeltagamma}, for all $i\geq 1$.
\end{Proposition}
\begin{proof} Fix a nonnegative integer $N$ and let $x = V_B^N(x_0) = (e_1,e_2,e_3,...)$, with $e_i=(s_i,m_i,r_i)$ for all $i\geq 1$. 

From Remark \ref{caminhos}, we have that for each nonnegative integer $k \geq 2$
\begin{equation} \label{vfnx0}
V_B^{F_{k-1}}(x_0)=(\underbrace{\vec{0},\ldots,\vec{0}}_{(k-2) \textrm{ times}},\tilde{e},\tilde{f},\vec{0},\vec{0},\ldots) 
\end{equation}
where either $\tilde{e} = \vec{0}$ and $\tilde{f}=(1,1,1)$ if $a>1$ or $\tilde{e} = (1,0,2)$ and $\tilde{f}=(2,1,1)$ if $a=1$. Thus, since $e_k=(s_k,m_k,r_k)$, it follows that if $s_k = 1$ and $r_k = 1$, then $a> 1$, $m_k\in\{0,\ldots,a-1\}$ and  
\begin{eqnarray}
\label{eq:VFG1}
V_B^{\delta_kF_{k-1}+\gamma_kG_{k-1}}(x_0) & = & V_B^{m_k F_{k-1}}(x_0) \nonumber \\
& = & (\underbrace{\vec{0},\ldots,\vec{0}}_{(k-2) \textrm{ times}},\tilde{e},(1,m_k,1),\vec{0},\vec{0},\ldots) \, ,
\end{eqnarray}
and if $s_k = 2$ and $r_k = 1$, then $m_k\in\{a,\ldots,a+b-1\}$ and  
\begin{eqnarray}
\label{eq:VFG2}
V_B^{\delta_kF_{k-1}+\gamma_kG_{k-1}}(x_0) & = & V_B^{aF_{k-1}+(m_k-a)G_{k-1}}(x_0) \nonumber \\
& = & (\underbrace{\vec{0},\ldots,\vec{0}}_{(k-2) \textrm{ times}},\tilde{e},(2,m_k,1),\vec{0},\vec{0},\ldots) \, .
\end{eqnarray}

Now, consider $k = \xi(x)= \min \{ j \ge 1: e_l = \vec{0} \ \textrm{ for all } \ l> j \} \,$ and put $N_k=\delta_k F_{k-1}+\gamma_k G_{k-1}$. For each $j\in\{1,\ldots,k-1\}$, let $N_j=\delta_jF_{j-1}+\gamma_jG_{j-1}+N_{j+1}$ and $x(j+1):=(\underbrace{\vec{0},\ldots,\vec{0}}_{j-1 \textrm{ times}},\tilde{e},e_{j+1},e_{j+2},\ldots,e_k,\vec{0},\vec{0},\ldots)$. Suppose that $V_B^{N_{j+1}}(x_0)=x(j+1)$, for some $j\in\{1,\ldots,k-1\}$.

Here, we need to consider four cases:
$$\begin{array}{rlrl}
i) & s_{j}=1  \textrm{ and } r_j= 1;  & iii) & s_{j}=1  \textrm{ and } r_j= 2;     \\ 
ii) & s_{j}=2  \textrm{ and } r_j= 1; & iv) & s_{j}=2  \textrm{ and } r_j= 2.     \\
\end{array}
$$

For example, in the case $ii)$ we have $\tilde{e}=(1,0,1)$, $m_j\in\{a,\ldots,a+b-1\}$ and  $V_B^{N_j}(x_0)= V_B^{aF_{j-1}+(m_j-a)G_{j-1}}\left(V_B^{N_{j+1}}(x_0)\right) = V_B^{aF_{j-1}+(m_j-a)G_{j-1}}(x(j+1))=x(j)$. In the same way, we can check that $V_B^{N_j}(x_0)=x(j)$ for the other cases.

By induction we have $V_B^{N_1}(x_0)=x(1) = x$ and since $\delta_i=\gamma_i=0$, for all $i>k$, it follows that $N_1=\displaystyle \sum_{i=1}^k(\delta_iF_{i-1}+\gamma_iG_{i-1})=\sum_{i=1}^{+\infty}(\delta_iF_{i-1}+\gamma_iG_{i-1})=N$.
\end{proof}

\begin{Remark}
We believe that the last proposition is another formulation of Lemma 4 in \cite{bdm}, which gives a formula of the first entrance time map.
\end{Remark}

\begin{Remark}
We call $((\delta_1,\gamma_1),(\delta_2,\gamma_2),\ldots)$ the $(F,G)$-representation of $N$ and we put $N=\sum_{j=0}^{\xi(x)}\delta_{j+1} F_j+\gamma_{j+1} G_j=((\delta_1,\gamma_1),(\delta_2,\gamma_2),\ldots)$. The set of $(F,G)$-representations is recognized by a finite graph called automaton (see Figure \ref{transdutor}).  \label{fgrepresentation}
\end{Remark}

\begin{figure}[h!]
	\centering
	\includegraphics[scale=0.6]{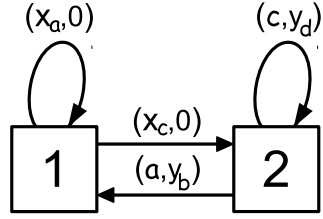}
	\caption{Automaton related to the $(F,G)$-representation of $N=((\delta_1,\gamma_1),(\delta_2,\gamma_2),\ldots)$, where $\delta_i\in\{x_a,x_c\}$ and $\gamma_i\in\{y_b,y_d\}$, for all $i\geq 1$ with $x_a\in\{0,\ldots,a-1\}$, $x_c\in\{0,\ldots,c-1\}$, $y_b\in\{0,\ldots,b-1\}$ and $y_d\in\{0,\ldots,d-1\}$.}\label{transdutor}
\end{figure}

\begin{Remark}
In Example \ref{exdeltagamma}, it follows by Proposition \ref{representation} that $L(N_x)=x$ and $L(N_y)=y$ where $N_x=F_0+2G_0+F_1+3G_1+F_2+G_2=65$ and $N_y=F_0+G_0+F_2+2G_2=69$.
\end{Remark}

\begin{Example}
If $M=(d)$ for $d\geq 2$, by Proposition \ref{representation}, we obtain the numeration in base $d$, with digits $\{0,1,\ldots,d-1\}$.
\end{Example}

\begin{Remark}
We can define the sequences $(\delta_i(x))_{i\geq 1}$ and $(\gamma_i(x))_{i\geq 1}$ for all $x\in \tilde{X}_B$ as done in Definition \ref{defdeltagamma}, in the case where the Bratteli diagram does not satisfies Hypothesis A, i.e. the incidence matrix $M=\left( \begin{array}{cc} a & b \\ 1 & 0 \end{array}\right)$, where $ab>0$. Furthermore, in this case $\delta_i\in\{0,\ldots,a\}$ and $\gamma_i\in\{0,\ldots,b-1\}$, for all $i\geq 1$. Moreover, by Lemma \ref{lemmaseq}, we have that $G_n=F_{n-1}$, for all $n\geq 1$. By Proposition \ref{representation}, the $(F,G)$-representation of $N$ is given by the automaton represented in Figure \ref{fgrepresentationc1d0}.
\end{Remark}

\begin{figure}[h!]
	\centering
	\includegraphics[scale=0.5]{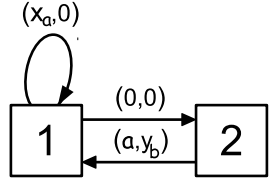}
	\caption{Automaton related to the $(F,G)$-representation, where $x_a\in\{0,\ldots,a-1\}$ and $y_b\in\{0,\ldots,b-1\}$.}\label{fgrepresentationc1d0}
\end{figure}

Observe in Figure \ref{fgrepresentationc1d0} that when $b=1$, the representation of $N$ is $((\delta_1,0)(\delta_2,0),\ldots)$, with $\delta_i\delta_{i-1}<_{lex} a1$, for all $i\geq 2$.

%%%%%%%%%%%%%%%%%%%%%%%%%%%%%%%%%%%%%%%%%%%%%%%%%%%%%%%%%%%%%%%%%%%%%%%%%%%%%%%%%%%%%%%%%%%%%%%
%%%%%%%%%%%%%%%%%%%%%%%%%%%%%%%%%%%%%%%%%%%%%%%%%%%%%%%%%%%%%%%%%%%%%%%%%%%%%%%%%%%%%%%%%%%%%%%

\subsection{Spectra of the stochastic machines of $2\times 2$ Bratteli diagrams}

We are finally in position to compute the spectrum of the transition operator (acting in $l^\infty(\mathbb{N})$) of the BV stochastic adding machines associated to a $2\times 2$ stationary Bratteli diagram endowed with the consecutive ordering. We denote the spectrum, point spectrum and approximate point spectrum of the transition operator $S$ respectively by $\sigma(S)$, $\sigma_{pt}(S)$ and $\sigma_a(S)$. Recall that $\lambda$ belongs to $\sigma(S)$ (resp. $\sigma_{pt}(S)$) if $S-\lambda I$ is not bijective (resp. not one-to-one). Also, $\lambda\in\sigma_a(S)$ if there exists a sequence $(v_n)_{n\geq 0}$ such that $\|v_n\|=1$, for all $n\geq 0$ and $(S-\lambda I)v_n$ converges to $0$ when $n$ goes to infinity.

For each $\lambda\in\mathbb{C}$, let $(u_{F_n}(\lambda))_{n\geq
0}=(u_{F_n})_{n\geq 0}$ and $(w_{F_n}(\lambda))_{n\geq 0}=(w_{F_n})_{n\geq 0}$ be the sequences defined by
$u_{F_0}=w_{F_0}=\frac{\lambda-(1-p_1)}{p_1}$ and for all $n\geq 1$.

\begin{equation}  \label{deffib}
u_{F_n}=\frac{1}{p_{n+1}}u_{F_{n-1}}^aw_{F_{n-1}}^b-\frac{1-p_{n+1}}{p_{n+1}} \textrm{, } w_{F_n}=\frac{1}{p_{n+1}}u_{F_{n-1}}^cw_{F_{n-1}}^d-\frac{1-p_{n+1}}{p_{n+1}}.
\end{equation}

\noindent From this, let $(v_n)_{n\geq 1}$ be the sequence defined by $v_n=\prod_{i=0}^{\xi(n)-1}u_{F_i}^{\delta_{i+1}}w_{F_i}^{\gamma_{i+1}}$, where $\delta_j=\delta_j(n)$ and $\gamma_j=\gamma_j(n)$, $j\in\{1,\ldots,\xi(n)\}$, are given in Definition \ref{defdeltagamma}. Since $v_{F_n}=u_{F_n}$, for all $n\geq 0$, we will denote $v_n$ by $u_n$.

\begin{Theorem} Let $S$ be the transition operator of a BV stochastic machine associated to a $2 \times 2$ Bratteli diagram $B$. Then, acting in $l^\infty(\mathbb{N})$, we have that the set of eigenvalues of $S$ is
\begin{center}
$\sigma_{pt}(S)=\{\lambda\in\mathbb{C}:(u_n(\lambda))_{n\geq 1}
\textrm{ is bounded}\}$.
\end{center}
\label{teorema431}
\end{Theorem}

\begin{Remark} \label{obsfib}
From Theorem \ref{teorema431}, we deduce that $$\sigma_{pt}(S)\subset\{\lambda\in\mathbb{C}: (u_{F_n}(\lambda))_{n\geq 0} \textrm{ is bounded}\}.$$ Moreover, if $\det M\leq 0$, we can show (see Proposition \ref{detmenor}) that $$\sigma_{pt}(S)\subset\mathcal{E}:=\{\lambda\in\mathbb{C}: (u_{F_n}(\lambda),w_{F_n}(\lambda)) \textrm{ is bounded}\}.$$	Since, $g_n\circ\ldots\circ g_1(u_{F_0},u_{F_0})=(u_{F_n},w_{F_n})$, for all $n\geq 1$, where $g_n:\mathbb{C}^2\longrightarrow\mathbb{C}^2$ are polynomials defined by 
	\begin{center}
		$g_n(x,y)=\left(\frac{1}{p_{n+1}}x^ay^b-\frac{1-p_{n+1}}{p_{n+1}}, \frac{1}{p_{n+1}}x^cy^d-\frac{1-p_{n+1}}{p_{n+1}}\right)$, for all $n\geq 1$,
	\end{center}

\noindent if follows that $\sigma_{pt}(S)$ is contained in the set $\{\lambda\in\mathbb{C}:(\frac{\lambda-1+p_1}{p_1},\frac{\lambda-1+p_1}{p_1})\in \mathcal{K}\}$ where $$\mathcal{K}:=\{(x,y)\in\mathbb{C}^2: (g_n\circ\ldots\circ g_1(x,y))_{n\geq 1} \textrm{ is bounded} \}.$$ This set is the 2-dimensional fibered filled Julia set associated to $(g_n)_{n\ge 1}$ (for more on fibered Julia sets see \cite{s} and references therein). In particular, if $(p_i)_{i\geq 1}$ is constant, then $\mathcal{K}$ is a 2-dimensional filled Julia set.
\end{Remark}

For the proof of Theorem \ref{teorema431}, we need the following lemma.

\begin{Lemma} \label{36}
For all $z=(z_i)_{i\geq 0}\in l^\infty$,
\begin{center}
		$(Sz)_N=\displaystyle\left(\prod_{j=1}^{\zeta_N}p_j\right)z_{N+1}+(1-p_1)z_N +\displaystyle\sum_{r=1}^{\zeta_N-1}\left(\prod_{j=1}^{r}p_j\right)(1-p_{r+1})z_{N-\sum_{j=0}^{r-1}\delta_{j+1}F_j+\gamma_{j+1}G_j}$,
\end{center}
if $\zeta_N\geq 2$ and $(Sz)_N=p_1z_{N+1}+(1-p_1)z_N$ if $\zeta_N=1$, where $\delta_i,\gamma_i$ are given in Definition \ref{defdeltagamma}, for all $i\in\{1,\ldots,\zeta_N\}$. 
\end{Lemma}
\begin{proof}
Let $N\in\mathbb{N}$ and $V^N_B(x_0)=x=(e_1,e_2,e_3,\ldots)$. All we need to do is identify $S_{N, \tilde{N}}$ for $\tilde{N} \in \mathbb{N}$.

Let $\xi(x)=k$ and $\zeta(x)=\zeta_N$. Thus, $x=(e_1,\ldots,e_{\zeta_N-1},e_{\zeta_N},e_{\zeta_N+1},\ldots,e_k,\vec{0},\vec{0},\ldots)$ and under Hypothesis A, we have that $A(x)=\{1,\ldots,\zeta_N-1\}$.

From Definition \ref{defvbs} and Remark \ref{remark:condA}, we have that $S_{N,N}=1-p_1$, $S_{N,N+1}=\prod_{j=1}^{\zeta_N}p_j$ and $S_{N,\tilde{N}}=0$ if $\tilde{N}\notin\{N,N+1\}$.
Thus, if $\zeta_N=1$, we are done. Suppose that $\zeta_N\geq 2$.

For each $i\in A(x)$, consider $y_i(x)$ defined by relation (\ref{yjx}). We can check that
\begin{center} 
$y_i(x)=(\underbrace{\vec{0},\ldots,\vec{0}}_{i-1 \textrm{ times}},\tilde{e},e_{i+1},e_{i+2},\ldots,e_{\zeta_N-1},e_{\zeta_N},e_{\zeta_N+1},\ldots,e_k,\vec{0},\vec{0},\ldots)$,
\end{center}

\noindent where $\tilde{e}=(1,0,1)=\vec{0}$ if $s_{i+1}=1$ and $\tilde{e}=(1,0,2)$ if $s_{i+1}=2$.

For each $i\in A(x)$, let $N_i\in \mathbb{N}$, such that $V^{N_i}_B(x_0)=y_i(x)$. Thus, from Proposition \ref{representation}, we have that $N_i=N-\sum_{j=0}^{i-1}\delta_{j+1}F_j+\gamma_{j+1}G_j$. Hence, from Remark \ref{remark:condA}, we have that $S_{N,N_i}=\prod_{j=1}^ip_j(1-p_{i+1})$. Furthermore $S_{N,\tilde{N}} = 0$ if $\tilde{N} \notin \{N,N+1,N_i, i \in A(x)\}$ and the proof is finished.
\end{proof}

Our next step is to prove Theorem \ref{teorema431}. The proof uses the same idea of the case $M=(d)$, for $d\geq 2$ done in \cite{msv}. However, the extension is far from elementary.

\begin{proof}[Proof of Theorem \ref{teorema431}] Let $z=(z_N)_{N\geq 0}$ be a sequence of complex numbers such that $(Sz)_N=\lambda z_N$ for every $N\geq 0$. We shall prove that $z_N=u_N z_0$ for all $N\geq 1$. For this we need to have in mind the representation of $N$ as a path in $\widetilde{X}_B$, i.e. $x = V^N(x_0) = (e_1,...,e_{\xi(x)},\vec{0},\vec{0},...)$ where $e_j = (s_j,m_j,r_j)$, $1\le j \le \xi(x)$.

The proof is based on the representation of Lemma \ref{36}. We use induction on $N\in\mathbb{N}$.

For $N=1$ we have by definition that $\delta_1=1$, $\gamma_1=0$ and $\delta_j=\gamma_j=0$ for all $j\geq 2$. Furthermore,
\begin{center}
$\lambda z_0=(1-p_1)z_0+p_1z_1$ $\Rightarrow$ $z_1=\displaystyle\dfrac{\lambda-1+p_1}{p_1}z_0=u_1z_0$.
\end{center}

Now fix $N=((\delta_1,\gamma_1),(\delta_2,\gamma_2),\ldots)\geq1$ and suppose that $z_j=u_jz_0$ for all $j\in\{1,\ldots,N\}$. Suppose that $\zeta_N=1$. Since
\begin{center}
$z_N=z_0\prod_{i=0}^{\xi (x)-1}u_{F_i}^{\delta_{i+1}}w_{F_i}^{\gamma_{i+1}}=
u_{F_0}^{\delta_1}w_{F_0}^{\gamma_1}z_0\prod_{i=1}^{\xi(x)-1}u_{F_i}^{\delta_{i+1}}w_{F_i}^{\gamma_{i+1}}$,
\end{center}
\noindent $(Sz)_N=p_1z_{N+1}+(1-p_1)z_N=\lambda z_N$ and

\begin{equation}
u_{F_0}=w_{F_0}=\frac{\lambda-(1-p_1)}{p_1}, \label{uzerowzero}
\end{equation}
we have
\begin{center}
$z_{N+1}=u_{F_0}z_N =
u_{F_0}^{\delta_1+\gamma_1+1}z_0\prod_{i=1}^{\xi(x)-1}u_{F_i}^{\delta_{i+1}}w_{F_i}^{\gamma_{i+1}}$.
\end{center}

From here, we need to consider two cases:

\smallskip

\noindent \textbf{Case 1}: if $s_1=1$, then $0\leq \delta_1< a$ if $r_1=1$ and $0\leq \delta_1<c$ if $r_1 = 2$. Furthermore, $\gamma_1=0$. Thus, $N+1=((\delta_1+1,\gamma_1),(\delta_2,\gamma_2),\ldots)$ and $$u_{N+1}=u_{F_0}^{\delta_1+1}w_{F_0}^{\gamma_1}\prod_{i=1}^{\xi(x)-1}u_{F_i}^{\delta_{i+1}}w_{F_i}^{\gamma_{i+1}}= u_{F_0}^{\delta_1+\gamma_1+1}\prod_{i=1}^{\xi(x)-1}u_{F_i}^{\delta_{i+1}}w_{F_i}^{\gamma_{i+1}}.$$

\smallskip

\noindent \textbf{Case 2}: if $s_1=2$, then $\delta_1=a$ and $0\leq \gamma_1<b-1$ if $r_1=1$ and $\delta_1=c$ and $0\leq \gamma_1<d-1$ if $r_1=2$. Thus, $N+1=((\delta_1,\gamma_1+1),(\delta_2,\gamma_2),\ldots)$ and $$u_{N+1}=u_{F_0}^{\delta_1}w_{F_0}^{\gamma_1+1}\prod_{i=1}^{\xi(x)-1}u_{F_i}^{\delta_{i+1}}w_{F_i}^{\gamma_{i+1}}= u_{F_0}^{\delta_1+\gamma_1+1}\prod_{i=1}^{\xi(x)-1}u_{F_i}^{\delta_{i+1}}w_{F_i}^{\gamma_{i+1}}.$$

Hence, in both cases we have that $z_{N+1}=u_{N+1}z_0$.

Now for $\zeta_N \ge 2$ we consider separately the cases $d>0$ and $d=0$.

\smallskip 

\noindent \textbf{Case d>0:}

\smallskip

First, suppose that $\zeta_N=2$, i.e. $e_1=(s_1,m_1,r_1)$ is a maximal edge and $e_2$ is not maximal. Thus, by Lemma \ref{36} and the fact that $(Sz)_N=\lambda z_N$, we have

\begin{center}
	$z_{N+1}=\displaystyle \frac{1}{p_1p_2}\left((\lambda-(1-p_1))z_N - p_1(1-p_2)z_{N-\delta_1F_0-\gamma_1 G_0}\right)$.
\end{center}

Hence,

\begin{center}
$\dfrac{z_{N+1}}{z_0\prod_{r=2}^{\xi(x)-1}u_{F_r}^{\delta_{r+1}}w_{F_r}^{\gamma_{r+1}}}=
\dfrac{(\lambda-(1-p_1)) u_{F_0}^{\delta_1}w_{F_0}^{\gamma_1}u_{F_1}^{\delta_2} w_{F_1}^{\gamma_2}}{p_1p_2}- \dfrac{1-p_2}{p_2}u_{F_1}^{\delta_2}w_{F_1}^{\gamma_2}$.
\end{center}

Since $e_1$ is a maximal edge, it follows that $s_1=2$. If $r_1=1$, then $\delta_1=a$ and $\gamma_1=b-1$ and if $r_1=2$ then $\delta_1=c$ and $\gamma_1 = d-1$. Thus,

\begin{center}
$\dfrac{z_{N+1}}{z_0\prod_{r=2}^{\xi(x)-1}u_{F_r}^{\delta_{r+1}}w_{F_r}^{\gamma_{r+1}}}=\left\{\begin{array}{l}
\dfrac{\lambda-(1-p_1)}{p_1}\cdot\dfrac{u_{F_0}^aw_{F_0}^{b-1}u_{F_1}^{\delta_2} w_{F_1}^{\gamma_2}}{p_2}- \dfrac{1-p_2}{p_2}u_{F_1}^{\delta_2}w_{F_1}^{\gamma_2}  \textrm{, if } r_1 = 1 , \\
\dfrac{\lambda-(1-p_1)}{p_1}\cdot\dfrac{u_{F_0}^cw_{F_0}^{d-1}u_{F_1}^{\delta_2} w_{F_1}^{\gamma_2}}{p_2}- \dfrac{1-p_2}{p_2}u_{F_1}^{\delta_2}w_{F_1}^{\gamma_2} \textrm{, if } r_1 = 2 .\\
\end{array}\right.$\end{center}
By (\ref{uzerowzero}), we deduce
\begin{center}
$\dfrac{z_{N+1}}{z_0\prod_{r=2}^{\xi(x)-1}u_{F_r}^{\delta_{r+1}}w_{F_r}^{\gamma_{r+1}}}=\left\{\begin{array}{l}
\left(\dfrac{1}{p_2} u_{F_0}^{a+b}- \dfrac{1-p_2}{p_2}\right)u_{F_1}^{\delta_2}w_{F_1}^{\gamma_2}=u_{F_1}^{\delta_2+1} w_{F_1}^{\gamma_2}  \textrm{, if } r_1 = 1 , \\
\left(\dfrac{1}{p_2}w_{F_0}^{c+d}- \dfrac{1-p_2}{p_2}\right)u_{F_1}^{\delta_2}w_{F_1}^{\gamma_2}=u_{F_1}^{\delta_2} w_{F_1}^{\gamma_2+1} \textrm{, if } r_1 = 2, \\
\end{array}\right.$\end{center}
\noindent and so
\begin{equation}
\label{zn+1}
z_{N+1}=
	\left\{\begin{array}{l}
	z_0u_{F_1}^{\delta_2+1} w_{F_1}^{\gamma_2}\prod_{r=2}^{\xi(x)-1}u_{F_r}^{\delta_{r+1}}w_{F_r}^{\gamma_{r+1}} \textrm{, if } r_1 = 1,\\
	z_0u_{F_1}^{\delta_2} w_{F_1}^{\gamma_2+1}\prod_{r=2}^{\xi(x)-1}u_{F_r}^{\delta_{r+1}}w_{F_r}^{\gamma_{r+1}} \textrm{, if } r_1 = 2.\\
	\end{array}\right.
\end{equation}
Since
$$N=((\delta_1,\gamma_1)(\delta_2,\gamma_2)\ldots)=\left\{\begin{array}{cc}
((a,b-1)(\delta_2,\gamma_2)\ldots ) \textrm{, if } r_1=1, \\
((c,d-1)(\delta_2,\gamma_2)\ldots ) \textrm{, if } r_1=2, \\
\end{array}\right.$$
it follows that
$$N+1=\left\{\begin{array}{cc}
((0,0)(\delta_2+1,\gamma_2)\ldots) \textrm{, if } r_1=1, \\
((0,0)(\delta_2,\gamma_2+1)\ldots) \textrm{, if } r_1=2, \\
\end{array}\right.$$
and from \eqref{zn+1}, we have that $z_{N+1}=u_{N+1}z_0$.

\smallskip

Finally we have to consider $\zeta_N \geq 3$. In this case, since $(e_1,\ldots,e_{\zeta_N-1})$ is a maximal element of $E(1)\circ E(2)\circ\ldots\circ E(\zeta_N-1)$ and $d>0$, it follows that $s_j=r_j=2$ for all $j\in\{1,\ldots,\zeta_N-2\}$. Therefore, $m_j=c+d-1$ (i.e $\delta_j=c$ and $\gamma_j=d-1$) for all $j\in\{1,\ldots,\zeta_N-2\}$. Furthermore, we have two subcases: 
\begin{enumerate}
\item[(1)] if $r_{\zeta_N-1} = 1$ then $m_{\zeta_N-1}=a+b-1$ (i.e $\delta_{\zeta_N-1}=a$ and $\gamma_{\zeta_N-1}=b-1$),  
\item[(2)] if $r_{\zeta_N-1}=2$ then $m_{\zeta_N-1}=c+d-1$ (i.e $\delta_{\zeta_N-1}=c$ and $\gamma_{\zeta_N-1}=d-1$).
\end{enumerate}
Thus, by Lemma \ref{36} and Definition \ref{defdeltagamma}, since $(Sz)_N=\lambda z_N$, we have that
\begin{equation}
\dfrac{z_{N+1}}{z_0\prod_{r=\zeta_N}^{\xi(x)-1}u_{F_r}^{\delta_{r+1}}w_{F_r}^{\gamma_{r+1}}}= \label{3.7}
\end{equation}
\begin{equation}
\dfrac{[\lambda-(1-p_1)]\left[\prod_{r=0}^{\zeta_N-3}u_{F_r}^cw_{F_r}^{d-1}\right]u_{F_{\zeta_N-2}}^{\delta_{\zeta_N-1}} w_{F_{\zeta_N-2}}^{\gamma_{\zeta_N-1}}u_{F_{\zeta_N-1}}^{\delta_{\zeta_N}} w_{F_{\zeta_N-1}}^{\gamma_{\zeta_N}}}{\prod_{j=1}^{\zeta_N}p_j}- \label{3.8}
\end{equation}
\begin{center}
$\dfrac{(1-p_2)\left[\prod_{r=1}^{\zeta_N-3}u_{F_r}^cw_{F_r}^{d-1}\right]u_{F_{\zeta_N-2}}^{\delta_{\zeta_N-1}} w_{F_{\zeta_N-2}}^{\gamma_{\zeta_N-1}}u_{F_{\zeta_N-1}}^{\delta_{\zeta_N}} w_{F_{\zeta_N-1}}^{\gamma_{\zeta_N}}}{\prod_{j=2}^{\zeta_N}p_j}-\ldots- $
\end{center}
\begin{center}
$\dfrac{(1-p_{\zeta_N-1})u_{F_{\zeta_N-2}}^{\delta_{\zeta_N-1}} w_{F_{\zeta_N-2}}^{\gamma_{\zeta_N-1}}u_{F_{\zeta_N-1}}^{\delta_{\zeta_N}} w_{F_{\zeta_N-1}}^{\gamma_{\zeta_N}}}{\prod_{j=\zeta_N-1}^{\zeta_N}p_j} -\dfrac{1-p_{\zeta_N}}{p_{\zeta_N}}u_{F_{\zeta_N-1}}^{\delta_{\zeta_N}} w_{F_{\zeta_N-1}}^{\gamma_{\zeta_N}}$.
\end{center}

By (\ref{uzerowzero}), the first term in (\ref{3.8}) is equal to
\begin{center}
$\dfrac{w_{F_0}^{c+d}\left[\prod_{r=1}^{\zeta_N-3}u_{F_r}^cw_{F_r}^{d-1}\right]u_{F_{\zeta_N-2}}^{\delta_{\zeta_N-1}} w_{F_{\zeta_N-2}}^{\gamma_{\zeta_N-1}}u_{F_{\zeta_N}-1}^{\delta_{\zeta_N}} w_{F_{\zeta_N-1}}^{\gamma_{\zeta_N}}}{\prod_{j=2}^{\zeta_N}p_j}.$
\end{center}

Summing with the second term, we get

\begin{center}
$\dfrac{w_{F_0}^{c+d}-(1-p_2)}{p_2}\cdot\dfrac{\left[\prod_{r=1}^{\zeta_N-3}u_{F_r}^cw_{F_r}^{d-1}\right]u_{F_{\zeta_N-2}}^{\delta_{\zeta_N-1}} w_{F_{\zeta_N-2}}^{\gamma_{\zeta_N-1}} u_{F_{\zeta_N-1}}^{\delta_{\zeta_N}} w_{F_{\zeta_N-1}}^{\gamma_{\zeta_N}}}{\prod_{j=3}^{\zeta_N}p_j}$
\end{center}

\noindent which is equal to
\begin{eqnarray*}
\lefteqn{w_{F_1}\dfrac{\left[\prod_{r=1}^{\zeta_N-3}u_{F_r}^cw_{F_r}^{d-1}\right]u_{F_{\zeta_N-2}}^{\delta_{\zeta_N-1}} w_{F_{\zeta_N-2}}^{\gamma_{\zeta_N-1}}u_{F_{\zeta_N-1}}^{\delta_{\zeta_N}} w_{F_{\zeta_N-1}}^{\gamma_{\zeta_N}}}{\prod_{j=3}^{\zeta_N}p_j}
= } \\
& & u_{F_1}^cw_{F_1}^d\dfrac{\left[\prod_{r=2}^{\zeta_N-3}u_{F_r}^cw_{F_r}^{d-1}\right]u_{F_{\zeta_N-2}}^{\delta_{\zeta_N-1}} w_{F_{\zeta_N-2}}^{\gamma_{\zeta_N-1}} u_{F_{\zeta_N-1}}^{\delta_{\zeta_N}} w_{F_{\zeta_N-1}}^{\gamma_{\zeta_N}}}{\prod_{j=3}^{\zeta_N}p_j}.
\end{eqnarray*}
By induction, we have that the sum of the first $\zeta_N-1$ terms in (\ref{3.8}) is equal to
\begin{center}
$\dfrac{u_{F_{\zeta_N-2}}^{\delta_{\zeta_N-1}} w_{F_{\zeta_N-2}}^{\gamma_{\zeta_N-1}+1} u_{F_{\zeta_N-1}}^{\delta_{\zeta_N}} w_{F_{\zeta_N-1}}^{\gamma_{\zeta_N}}}{p_{\zeta_N}}$.
\end{center}
Finally, summing the previous expression with the last term in (\ref{3.8}) we have that (\ref{3.7}) is equal to
\begin{center}
$\left\{\begin{array}{l}
\dfrac{u_{F_{\zeta_N-2}}^aw_{F_{\zeta_N-2}}^b-(1-p_{\zeta_N})}{p_{\zeta_N}}u_{F_{\zeta_N-1}}^{\delta_{\zeta_N}} w_{F_{\zeta_N-1}}^{\gamma_{\zeta_N}}= u_{F_{\zeta_N-1}}^{\delta_{\zeta_N}+1} w_{F_{\zeta_N-1}}^{\gamma_{\zeta_N}} \textrm{, if } r_{\zeta_N - 1} = 1;\\
\dfrac{u_{F_{\zeta_N-2}}^cw_{F_{\zeta_N-2}}^d-(1-p_{\zeta_N})}{p_{\zeta_N}}u_{F_{\zeta_N-1}}^{\delta_{\zeta_N}} w_{F_{\zeta_N-1}}^{\gamma_{\zeta_N}}= u_{F_{\zeta_N-1}}^{\delta_{\zeta_N}} w_{F_{\zeta_N-1}}^{\gamma_{\zeta_N}+1} \textrm{, if } r_{\zeta_N - 1} = 2.
\end{array}\right.$
\end{center}

Therefore,

\begin{center}
$z_{N+1}=\left\{\begin{array}{l}
z_0u_{F_{\zeta_N-1}}^{\delta_{\zeta_N}+1}w_{F_{\zeta_N-1}}^{\gamma_{\zeta_N}}\prod_{r=\zeta_N}^{\xi(x)-1}u_{F_r}^{\delta_{r+1}}w_{F_r}^{\gamma_{r+1}} \textrm{, if } r_{\zeta_N - 1} = 1\\
z_0u_{F_{\zeta_N-1}}^{\delta_{\zeta_N}}w_{F_{\zeta_N-1}}^{\gamma_{\zeta_N}+1}\prod_{r=\zeta_N}^{\xi(x)-1}u_{F_r}^{\delta_{r+1}}w_{F_r}^{\gamma_{r+1}} \textrm{, if } r_{\zeta_N - 1} = 2\\
\end{array}\right.=u_{N+1}z_0$,
\end{center}

\noindent where the next equality comes from the fact that $\delta_i(N+1)=\gamma_i(N+1)=0$, for all $i\in\{1,\ldots,\zeta_N-1\}$.

\smallskip

\noindent \textbf{Case $d=0$:}

\smallskip

Suppose that $r_1=1$ and $\zeta_N$ is an odd number (the proof for the cases $r_1=2$ or $\zeta_N$ even can be dealt in the same way).

Thus, since $(e_1,\ldots,e_{\zeta_N-1})$ is a maximal element of $E(1)\circ E(2)\circ\ldots\circ E(\zeta_N-1)$, we have that $r_{2i-1} = 1$, $r_{2i} = 2$, $s_{2i} = 1$ and $s_{2i-1}= 2$, for all $i\in \{1,\ldots,\frac{\zeta_N-1}{2}\}$. Therefore, $m_{2i-1}=a+b-1$ (i.e $\delta_{2i-1}=a$ and $\gamma_{2i-1}=b-1$) and $m_{2i}=c-1$ (i.e $\delta_{2i}=c-1$ and $\gamma_{2i}=0$) for all $i\in \{1,\ldots,\frac{\zeta_N-1}{2}\}$.

For each $i\in \{0,\ldots,\zeta_N-2\}$, let $\mathbb{P}_i$ be the product defined by $\mathbb{P}_i:=\prod_{r=i}^{\zeta_N-2}u_{F_r}^{\delta{r+1}}w_{F_r}^{\gamma_{r+1}}$. Thus, we have either $\mathbb{P}_i=u_{F_i}^aw_{F_i}^{b-1}\mathbb{P}_{i+1}$ if $i$ is an even number or $\mathbb{P}_i=u_{F_i}^{c-1}\mathbb{P}_{i+1}$ if $i$ is an odd number.

By Lemma \ref{36}, since $(Sv)_N=\lambda v_N$, we have that

\begin{equation}
\dfrac{v_{N+1}}{v_0\prod_{r=\zeta_N}^{\xi(N+1)-1}u_{F_r}^{\delta_{r+1}}w_{F_r}^{\gamma_{r+1}}}= \label{4.4}
\end{equation}
\begin{equation}
\dfrac{[\lambda-(1-p_1)]\mathbb{P}_0u_{F_{\zeta_N-1}}^{\delta_{\zeta_N}} w_{F_{\zeta_N-1}}^{\gamma_{\zeta_N}}}{\prod_{j=1}^{\zeta_N}p_j}- \dfrac{(1-p_2)\mathbb{P}_1u_{F_{\zeta_N-1}}^{\delta_{\zeta_N}} w_{F_{\zeta_N-1}}^{\gamma_{\zeta_N}}}{\prod_{j=2}^{\zeta_N}p_j}-  \label{4.5}
\end{equation}
\begin{center}
	$\dfrac{(1-p_3)\mathbb{P}_2u_{F_{\zeta_N-1}}^{\delta_{\zeta_N}} w_{F_{\zeta_N-1}}^{\gamma_{\zeta_N}}}{\prod_{j=3}^{\zeta_N}p_j}-\ldots-  \dfrac{(1-p_{\zeta_N-2})\mathbb{P}_{\zeta_N-3}u_{F_{\zeta_N-1}}^{\delta_{\zeta_N}} w_{F_{\zeta_N-1}}^{\gamma_{\zeta_N}}}{\prod_{j=\zeta_N-2}^{\zeta_N}p_j}-$
\end{center}
\begin{center}
	$\dfrac{(1-p_{\zeta_N-1})\mathbb{P}_{\zeta_N-2}u_{F_{\zeta_N-1}}^{\delta_{\zeta_N}} w_{F_{\zeta_N-1}}^{\gamma_{\zeta_N}}}{\prod_{j=\zeta_N-1}^{\zeta_N}p_j} -\dfrac{1-p_{\zeta_N}}{p_{\zeta_N}}u_{F_{\zeta_N-1}}^{\delta_{\zeta_N}} w_{F_{\zeta_N-1}}^{\gamma_{\zeta_N}}$.
\end{center}

Since $u_{F_0}=w_{F_0}=\dfrac{\lambda-1+p_1}{p_1}$, the first term in (\ref{4.5}) is equal to
\begin{center}
	$\dfrac{u_{F_0}\mathbb{P}_0u_{F_{\zeta_N-1}}^{\delta_{\zeta_N}} w_{F_{\zeta_N-1}}^{\gamma_{\zeta_N}}}{\prod_{j=2}^{\zeta_N}p_j}= \dfrac{u_{F_0}^{a+b}\mathbb{P}_1u_{F_{\zeta_N-1}}^{\delta_{\zeta_N}} w_{F_{\zeta_N-1}}^{\gamma_{\zeta_N}}}{\prod_{j=2}^{\zeta_N}p_j}.$
\end{center}

Summing with the second term, we get
\begin{center}
	$\dfrac{u_{F_0}^{a+b}-(1-p_2)}{p_2}\dfrac{\mathbb{P}_1u_{F_{\zeta_N-1}}^{\delta_{\zeta_N}} w_{F_{\zeta_N-1}}^{\gamma_{\zeta_N}}}{\prod_{j=3}^{\zeta_N}p_j}=u_{F_1} \dfrac{\mathbb{P}_1u_{F_{\zeta_N-1}}^{\delta_{\zeta_N}} w_{F_{\zeta_N-1}}^{\gamma_{\zeta_N}}}{\prod_{j=3}^{\zeta_N}p_j}= u_{F_1}^c \dfrac{\mathbb{P}_2u_{F_{\zeta_N-1}}^{\delta_{\zeta_N}} w_{F_{\zeta_N-1}}^{\gamma_{\zeta_N}}}{\prod_{j=3}^{\zeta_N}p_j}$.
\end{center}

Summing with the third term, we get
\begin{center}
	$\dfrac{u_{F_1}^c-(1-p_3)}{p_3}\dfrac{\mathbb{P}_2u_{F_{\zeta_N-1}}^{\delta_{\zeta_N}} w_{F_{\zeta_N-1}}^{\gamma_{\zeta_N}}}{\prod_{j=4}^{\zeta_N}p_j}=w_{F_2} \dfrac{\mathbb{P}_2u_{F_{\zeta_N-1}}^{\delta_{\zeta_N}} w_{F_{\zeta_N-1}}^{\gamma_{\zeta_N}}}{\prod_{j=4}^{\zeta_N}p_j}= u_{F_2}^aw_{F_2}^b \dfrac{\mathbb{P}_3u_{F_{\zeta_N-1}}^{\delta_{\zeta_N}} w_{F_{\zeta_N-1}}^{\gamma_{\zeta_N}}}{\prod_{j=4}^{\zeta_N}p_j}$.
\end{center}

By induction we have that the sum of the first $\zeta_N-1$ terms in (\ref{4.5}) is equal to

\begin{center}
	$\dfrac{u_{F_{\zeta_N-3}}^aw_{F_{\zeta_N-3}}^b-(1-p_{\zeta_N-1})}{p_{\zeta_N-1}}\dfrac{\mathbb{P}_{\zeta_N-2}u_{F_{\zeta_N-1}}^{\delta_{\zeta_N}} w_{F_{\zeta_N-1}}^{\gamma_{\zeta_N}}}{\prod_{j=\zeta_N}^{\zeta_N}p_j}=u_{F_{\zeta_N-2}}^c \dfrac{u_{F_{\zeta_N-1}}^{\delta_{\zeta_N}} w_{F_{\zeta_N-1}}^{\gamma_{\zeta_N}}}{p_{\zeta_N}}$.
\end{center}

Finally, summing the previous expression with the last term in (\ref{4.5}) we have that (\ref{4.4}) is equal to
\begin{center}
	$\dfrac{u_{F_{\zeta_N-2}}^c-(1-p_{\zeta_N})}{p_{\zeta_N}}u_{F_{\zeta_N-1}}^{\delta_{\zeta_N}} w_{F_{\zeta_N-1}}^{\gamma_{\zeta_N}}= u_{F_{\zeta_N-1}}^{\delta_{\zeta_N}} w_{F_{\zeta_N-1}}^{\gamma_{\zeta_N}+1}$.
\end{center}

Therefore,

\begin{center}
	$v_{N+1}=v_0u_{F_{\zeta_N-1}}^{\delta_{\zeta_N}}w_{F_{\zeta_N-1}}^{\gamma_{\zeta_N}+1}\prod_{r=\zeta_N}^{\xi(N+1)-1}u_{F_r}^{\delta_{r+1}}w_{F_r}^{\gamma_{r+1}}=u_{N+1}v_0$,
\end{center}

\noindent where the last equality comes from the fact that $\delta_i(N+1)=\gamma_i(N+1)=0$, for all $i\in\{1,\ldots,\zeta_N-1\}$.
\end{proof}

\begin{Proposition} \label{proprop}
Let $\mathcal{F}:=\{\lambda\in\mathbb{C}:(u_{F_n}(\lambda))_{n\geq 0} \textrm{ is bounded}\}$. Then $\sigma_{pt}(S)\subset\mathcal{F}\subset\sigma_a(S)$.
\end{Proposition}
\begin{proof} By Theorem \ref{teorema431}, $\sigma_{pt}(S)\subset\mathcal{F}$ and we only have to prove that $\mathcal{F}\subset\sigma_a(S)$. 

Let $\lambda\in \mathcal{F}$ and suppose that
	$\lambda\notin\sigma_{pt}(S)$. We will prove that
	$\lambda\in\sigma_a(S)$. In fact, for each $k\geq 2$, consider
	
	\begin{center}
		$x^{(k)}=(x^{(k)}_0,x^{(k)}_1,x^{(k)}_2,\ldots,x^{(k)}_k,0,0,\ldots)=
		(1,u_1(\lambda),u_2(\lambda),\ldots,u_k(\lambda),0,0,\ldots)$,
	\end{center}
	
	\noindent where $(u_n(\lambda))_{n \geq 1}=(u_n)_{n \geq 1}$ is the
	sequence defined in relation (\ref{deffib}). Define $$y^{(k)}:=\frac{x^{(k)}}{\lVert
		x^{(k)}\rVert_{\infty}}.$$
	
	\noindent \textbf{Claim:} $\displaystyle\lim_{n\to+\infty}\lVert(S-\lambda
	I)y^{(F_n)}\rVert_\infty=0$.
	
	\noindent In fact, for all $i\in\{0,\ldots,k-1\}$, we have $((S-\lambda
	I)y^{(k)})_i=0$ and $y_i=0$, for all $i>k$. Hence, note that
	\begin{center}
		$\|(S-\lambda I)y^{(k)}\|_\infty= \displaystyle\sup_{i\geq
			0}\left|\sum_{j=0}^{+\infty}(S-\lambda I)_{ij}y_j^{(k)}\right|=
		\displaystyle\sup_{i\geq
			k}\left\{\frac{\left|\sum_{j=0}^k(S-\lambda
			I)_{ij}x_j^{(k)}\right|}{\| x^{(k)}\|_\infty}\right\}$.
	\end{center}
	
Let $n>1$, $k=F_n\geq 2$ and $i\geq k$. We consider two cases:

\smallskip
	
\noindent \textbf{Case $a>1$:}

\begin{itemize}
\item If $i=F_n$, then since $a>1$, by relation (\ref{vfnx0}) we have that  $V_B^{F_n}(x_0)=(\underbrace{\vec{0},\ldots,\vec{0}}_{n},(1,1,1),\vec{0},\vec{0},\ldots)$. Since $n>1$, it follows that $S_{i,j}=0$, for all $j\in\{0,\ldots, F_n-1\}$ and $S_{i,i}=1-p_1$. Therefore,
	$\left|\sum_{j=0}^{F_n}(S-\lambda
	I)_{ij}x_j^{(F_n)}\right|=|1-p_1-\lambda||u_{F_n}|$.
	
\item If $F_n<i\leq 2F_n-1$, then since $a>1$, by the proof of Proposition \ref{representation}, we have that  $V_B^i(x_0)=(e_1,\ldots,e_n,(1,1,1),\vec{0},\vec{0},\ldots)$. Hence $S_{i,j}=0$, for all
$j\in\{0,\ldots, F_n-1\}$. Furthermore, since $S_{i,i}=1-p_1$ and $S$ is a stochastic matrix, it follows that $S_{i,j}\leq p_1$, for $j=F_n$.
	Therefore, $\left|\sum_{j=0}^{F_n}(S-\lambda I)_{ij}x_j^{(F_n)}\right|\leq
	p_1|u_{F_n}|$.

\item  If $i\geq 2F_n$, then $V_B^i(x_0)=(e_1,\ldots,e_l,\vec{0},\vec{0},\ldots)$, with $e_l\neq \vec{0}$ and $l\geq n+1$. Since $a>1$, we have $m_l>0$ and so $S_{i,j}=0$, for all $j\in\{1,\ldots, F_n\}$. Furthermore, $$p_1\geq S_{i,0}=\left\{\begin{array}{cl}
p_1\ldots p_l(1-p_{l+1}), & \textrm{if } (e_1,\ldots,e_l) \textrm{ is a maximal path;} \\
0, & \textrm{if is not.}
\end{array}\right.$$ Therefore,
	$\left|\sum_{j=0}^{F_n}(S-\lambda I)_{ij}x_j^{(F_n)}\right|\leq
	p_1|x_0|=p_1$.
\end{itemize}

\noindent \textbf{Case $a=1$:}

\begin{itemize}
\item If $i=F_n$ then $S_{i,j}=0$, for all $j\in\{0,\ldots,
F_n-1\}$, and $S_{i,i}=1-p_1$. Therefore,
$\left|\sum_{j=0}^{F_n}(S-\lambda
I)_{ij}x_j^{(F_n)}\right|=|1-p_1-\lambda||u_{F_n}|$.

\item If $F_n<i<F_n+G_n-1$ then $S_{i,j}=0$, for all
$j\in\{0,\ldots, F_n-1\}$, and $S_{i,j}\leq p_1$, for $j=F_n$.
Therefore, $\left|\sum_{j=0}^{F_n}(S-\lambda I)_{ij}x_j^{(F_n)}\right|\leq
p_1|u_{F_n}|$.

\item If $i=F_n+G_n-1$ then $S_{i,j}=0$, for all
$j\in\{1,\ldots, F_n-1\}$, $S_{i,j}\leq p_1$ for $j=F_n$, $S_{i,0}\leq p_1$ if $b=1$ and $S_{i,0}=0$ if $b>1$. Therefore,
$\left|\sum_{j=0}^{F_n}(S-\lambda I)_{ij}x_j^{(F_n)}\right|\leq
p_1+p_1|u_{F_n}|$.

\item If $i\geq F_n+G_n$ then $S_{i,j}=0$, for all
$j\in\{1,\ldots, F_n\}$, and $S_{i,j}\leq p_1$, for $j=0$. Therefore,
$\left|\sum_{j=0}^{F_n}(S-\lambda I)_{ij}x_j^{(F_n)}\right|\leq
p_1$.
\end{itemize}

	Hence, from both cases it follows that
	\begin{equation}
	\|(S-\lambda I)y^{(F_n)}\|_\infty\leq
	\frac{|1-p_1-\lambda||u_{F_n}|+p_1|u_{F_n}|+p_1}{\|
		x^{(F_n)}\|_\infty}. \label{desi}
	\end{equation}
	Since $\lambda\in \mathcal{F}$ and $\lambda\notin\sigma_{pt}(S)$, it follows
	that $(u_{F_n})_{n\geq 0}$ is a bounded sequence and $(u_n)_{n\geq
		0}$ is not. Therefore, we have $\lim_{n\to+\infty}\lVert
	x^{(F_n)}\rVert_\infty=+\infty$, which implies from relation (\ref{desi}) that
	$\displaystyle\lim_{n\to+\infty}\lVert(S-\lambda
	I)y^{(F_n)}\rVert_\infty=0$. Therefore,
	$\lambda\in\sigma_a(S)\subset\sigma(S)$.
\end{proof}

\begin{Proposition}
If $\det M=ad-bc\leq0$, then $(u_{F_n})_{n\geq 0}$ is bounded if and only if $(w_{F_n})_{n\geq 0}$ is bounded. \label{detmenor}
\end{Proposition}
\begin{proof}
Let $R_n=p_{n+1}u_{F_n}+1-p_{n+1}$ and $S_n=p_{n+1}w_{F_n}+1-p_{n+1}$, for all $n\geq 0$. By \eqref{deffib}, we have that
$$R_{n+1}^c=S_{n+1}^aw_{F_n}^{bc-ad} \textrm{ and } S_{n+1}^b=R_{n+1}^du_{F_n}^{bc-ad}.$$
Since $ad-bc\leq 0$ and $(p_n)_{n\geq 1}$ is bounded, we obtain the result.
\end{proof}

\begin{question}
If $\det M>0$, is $(u_{F_n})_{n\geq 0}$ bounded equivalent to $(w_{F_n})_{n\geq 0}$ bounded?
\end{question}

\begin{Remark} \label{remark417}
From Remark \ref{obsfib} and Propositions \ref{proprop} and \ref{detmenor}, we have that if $\det M\leq 0$, then \begin{center}$\sigma_{pt}(S)\subset\mathcal{E}=\{\lambda\in\mathbb{C}:(\frac{\lambda-1+p_1}{p_1},\frac{\lambda-1+p_1}{p_1})\in \mathcal{K}\}\subset \sigma_a(S)$.\end{center}
\end{Remark}

\begin{Remark}
	If $e:=a+b=c+d$, then we have $F_n=G_n=e^n$, for all $n\geq 0$. In this case, the Vershik map is related to addition of $1$ in base $e\geq 2$, see Remark \ref{remark:uniq} and Example \ref{example:cantor}. For this class, it was proved in \cite{msv} that the point spectrum of $S$ is equal to the fibered filled Julia set of $f_n(x)=\frac{1}{p_{n+1}}x^{a+b}-\left(\frac{1}{p_{n+1}}-1\right)$. In the next proposition we will prove the same result cited below.
\end{Remark}

\begin{Proposition} \label{teoremaab=cd}
	If $a+b=c+d$ then $\sigma_{pt}(S)=\mathcal{E}$. Furthermore, $\mathcal{E}=\{\lambda\in\mathbb{C}:(f_n\circ\ldots\circ f_1(u_{F_0}))_{n\geq 1} \textrm{ is bounded}\}$, where $f_n(x)=\frac{1}{p_{n+1}}x^{a+b}-\left(\frac{1}{p_{n+1}}-1\right)$, for all $n\geq 1$.
\end{Proposition}
\begin{proof}
From Theorem \ref{teorema431} and Remark \ref{obsfib} we have that $\sigma_{pt}(S)\subset \mathcal{E}$.
	
Let $\lambda\in\mathcal{E}$. Since $a+b=c+d$, it follows from \eqref{deffib} that $u_{F_n}(\lambda)=w_{F_n}(\lambda)$, for all $\lambda\in\mathbb{C}$ and $n\geq 1$. Thus, it follows that $|u_{F_n}(\lambda)|,|w_{F_n}(\lambda)|\leq 1$ for all $n\geq 0$, indeed let $R>1$ be a real number such that $|u_{F_k}|=|w_{F_k}|>R$. Since $abc>0$ and $c+d>1$, it follows that $\min\{a+b,c+d\} > 1$. Thus,
$|u_{F_{k+1}}|=|w_{F_{k+1}}|>\frac{1}{p_{k+2}}R^{a+b} -\frac{1-p_{k+2}}{p_{k+2}}>R^{a+b}\geq R^2$.

By induction we obtain that $|u_{F_{k+i}}|=|w_{F_{k+i}}|>R^{2^i}$, for all $i\geq 1$. Since $R>1$, it follows that $(u_{F_n})_{n\geq 0}$ and $(w_{F_n})_{n\geq 0}$ are unbounded and $\lambda \notin \mathcal{E}$ which yields a contradiction.

Therefore, if $\lambda\in\mathcal{E}$ and then $|u_{F_n}(\lambda)|,|w_{F_n}(\lambda)|\leq 1$ for all $n\geq 0$, by \eqref{deffib}, we have that $|u_n(\lambda)|\leq 1$, for all $n\geq 1$, i.e. $\lambda\in\sigma_{pt}(S)$.
	
To prove that $\mathcal{E}=\{\lambda\in\mathbb{C}:(f_n\circ\ldots\circ f_1(u_{F_0}))_{n\geq 1} \textrm{ is bounded}\}$, we just need to observe that $f_n\circ\ldots\circ f_1(u_{F_0}(\lambda))=u_{F_n}(\lambda)=w_{F_n}(\lambda)$, for all $n\geq 1$.
\end{proof}

\begin{Remark}
In Proposition \ref{teoremaab=cd} we have proved that if $\min\{|u_{F_i}|,|w_{F_i}|\}> 1$ for some integer $i$, then $\min\{|u_{F_n}|,|w_{F_n}|\}$ goes to $\infty$ when $n$ goes to infinity.
\end{Remark}

\begin{question}
If $\det M\leq 0$, can we prove that $\sigma_{pt}(S)=\mathcal{E}=\sigma(S)$?
\end{question}

\begin{Example} 
	Consider the consecutive ordering Bratteli diagram $B$ represented by the matrix $M=\left(\begin{array}{cc} 3 & 1 \\ 1 & 2 \end{array}\right)$. Below, we present some pictures describing the set $\mathcal{E}=\{\lambda\in\mathbb{C}: (u_{F_n}(\lambda),w_{F_N}(\lambda))_{n\geq 0} \textrm{ is bounded}\}$ for some choices of $(p_i)_{i\ge 1}$.
	\begin{figure}[h!]
		\centering
		\includegraphics[scale=0.37]{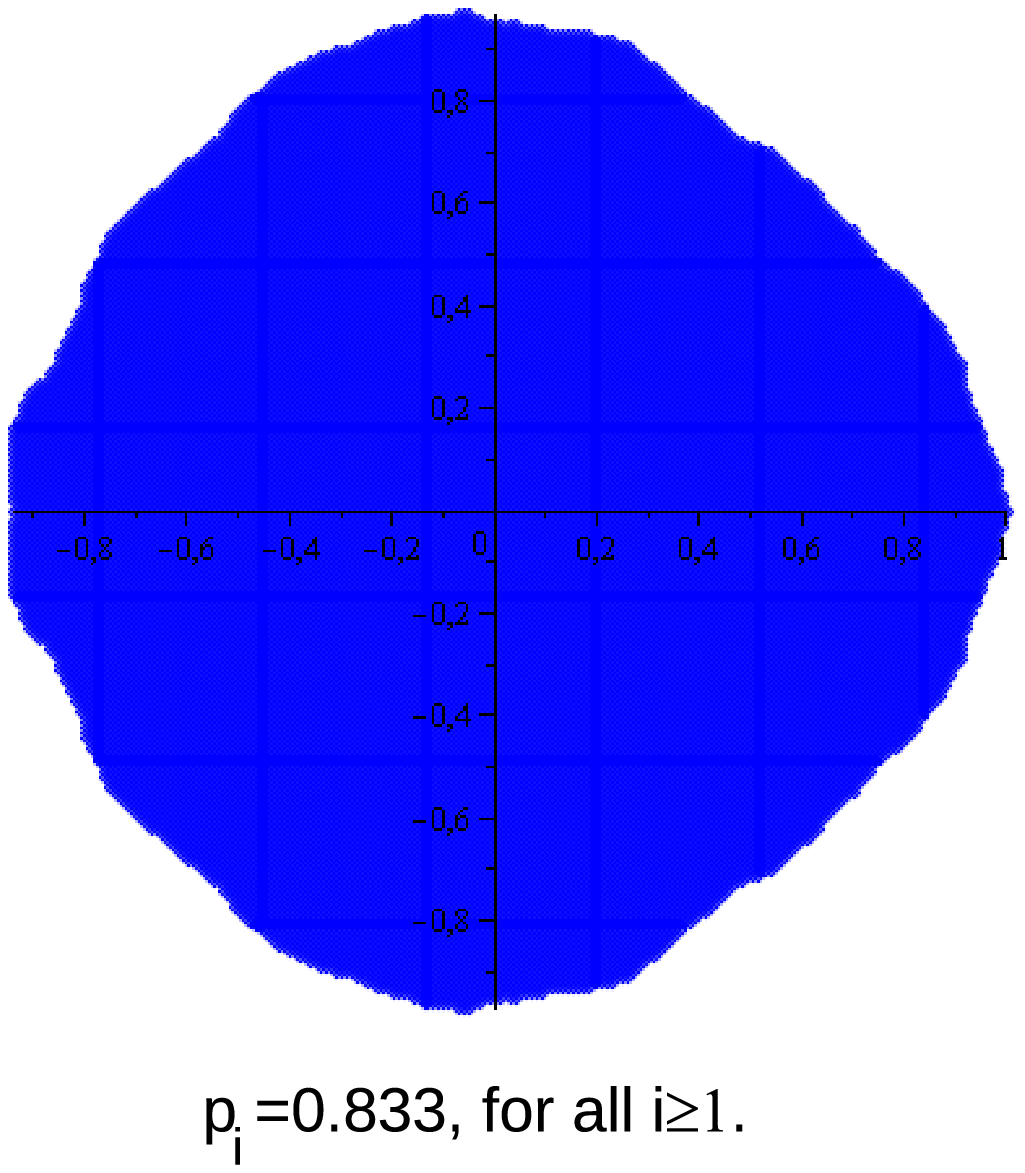}
		\includegraphics[scale=0.37]{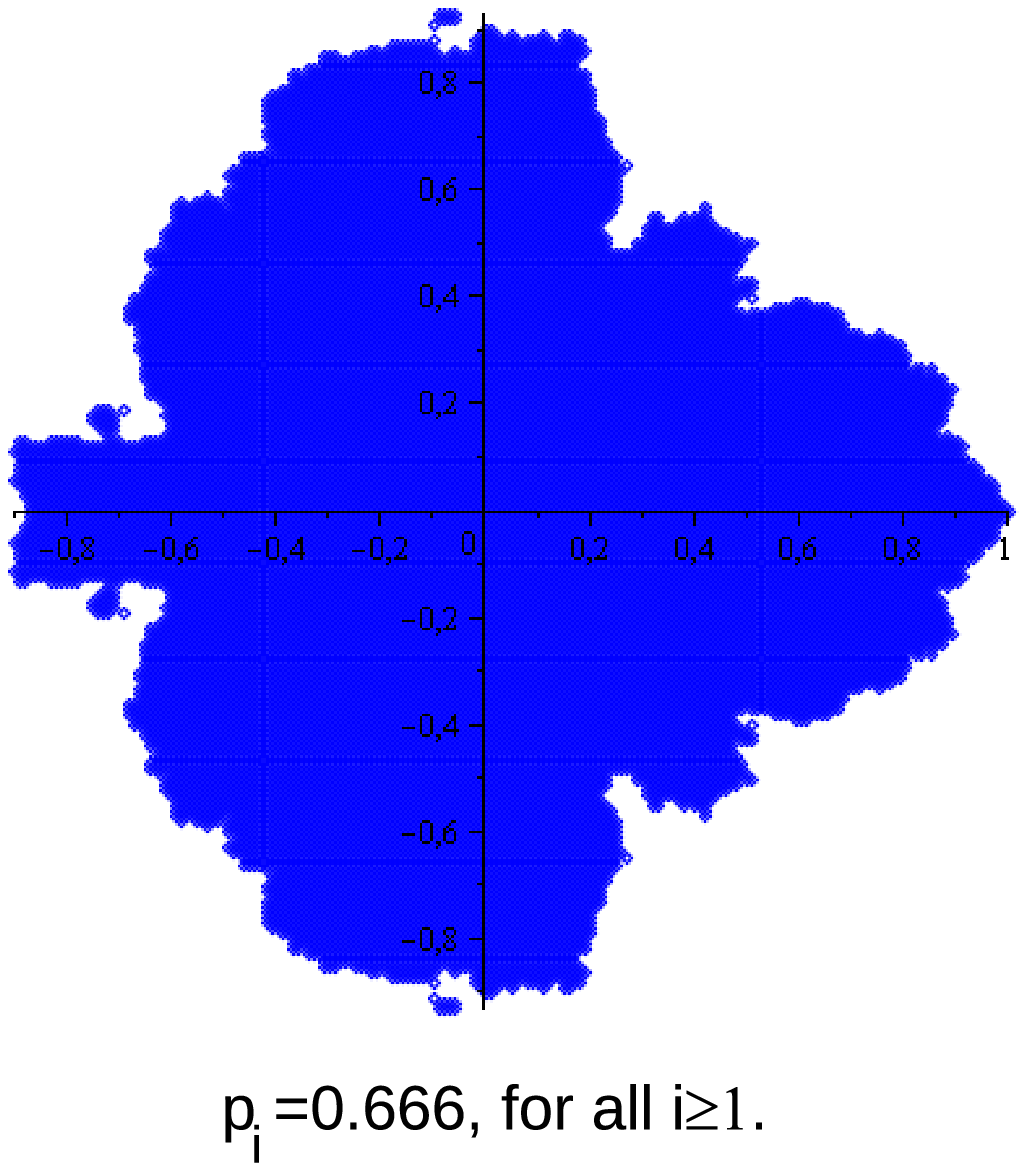}
		\includegraphics[scale=0.37]{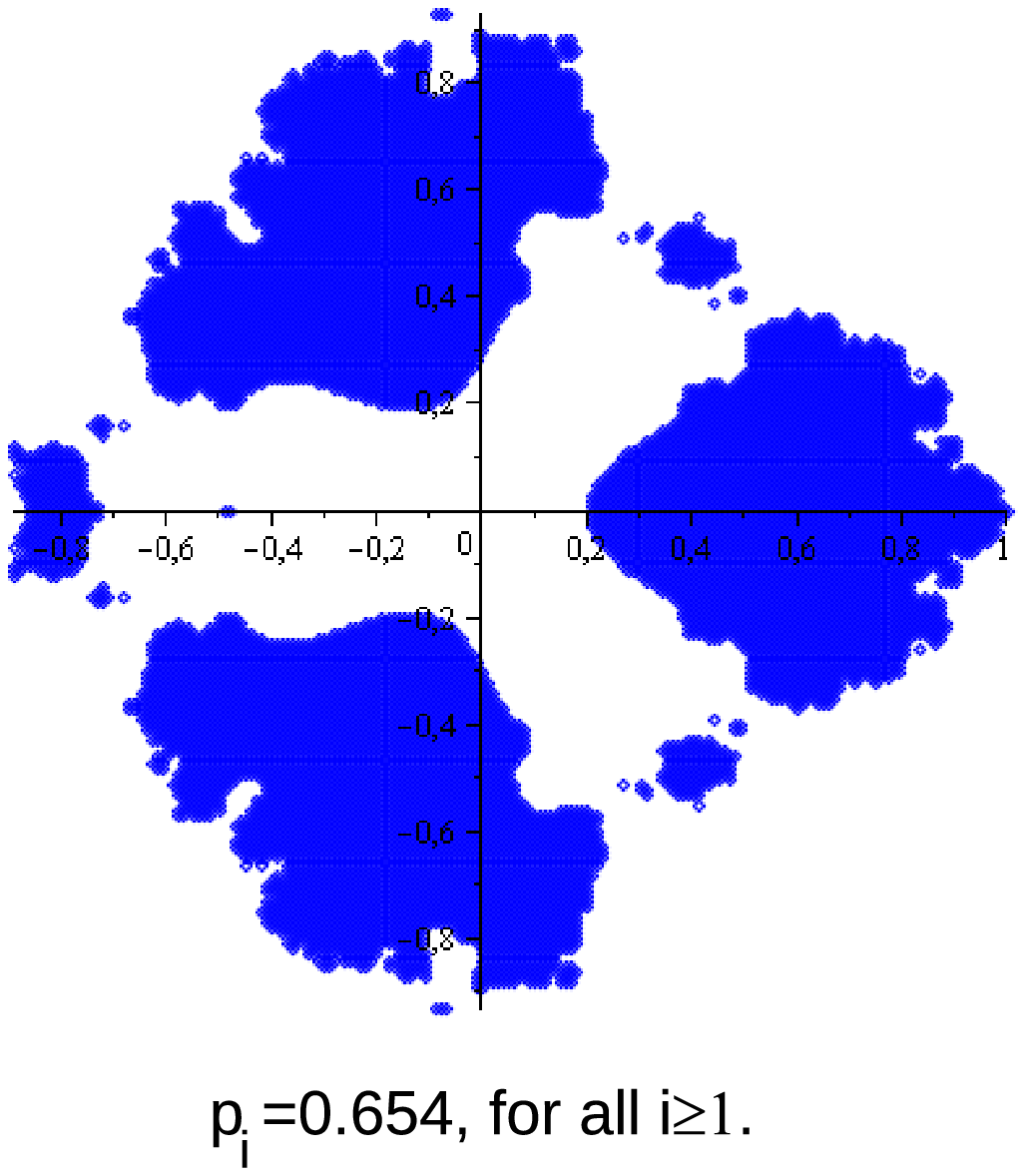}
	\end{figure}
\end{Example}

\begin{Example} 
	Consider the consecutive ordering Bratteli diagram $B$ represented by the matrix $M=\left(\begin{array}{cc} 2 & 1 \\ 3 & 1 \end{array}\right)$. Below, we present some pictures describing the set $\mathcal{E}=\mathcal{F}=\{\lambda\in\mathbb{C}: (u_{F_n}(\lambda))_{n\geq 0} \textrm{ is bounded}\}$ for some choices of $(p_i)_{i\ge 1}$.
	\begin{figure}[h!]
	\centering
	\includegraphics[scale=0.37]{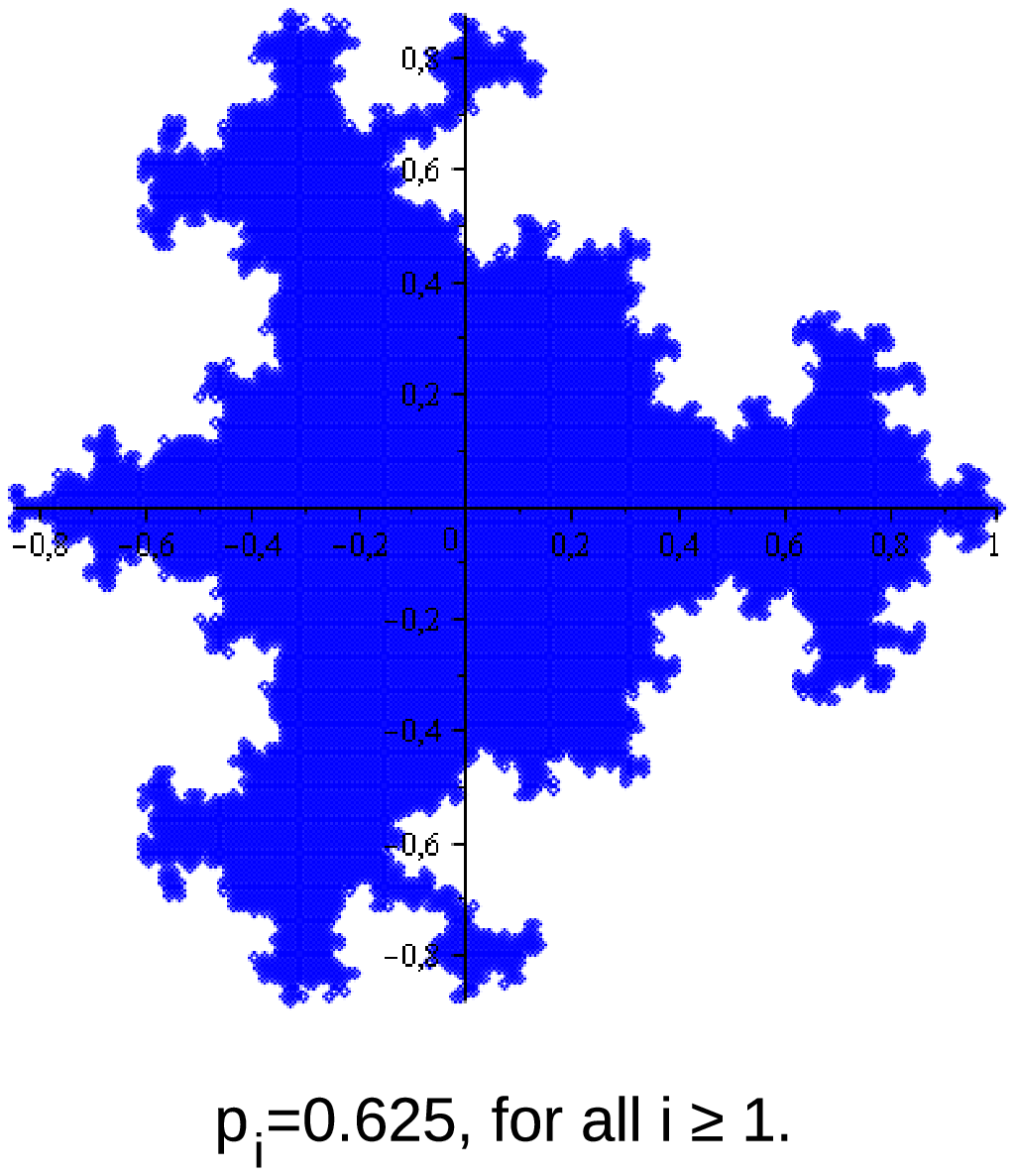}
	\includegraphics[scale=0.37]{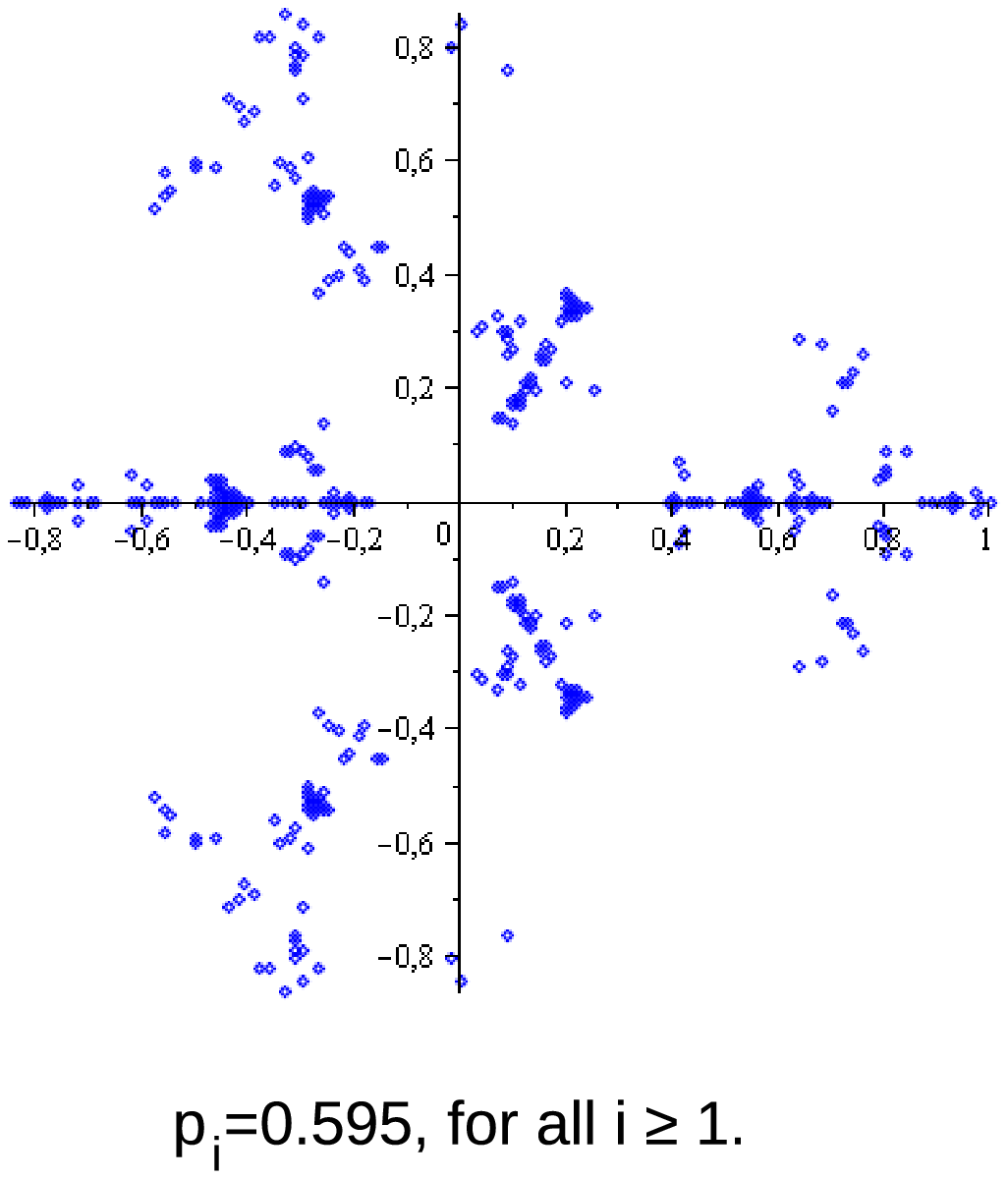}
	\includegraphics[scale=0.37]{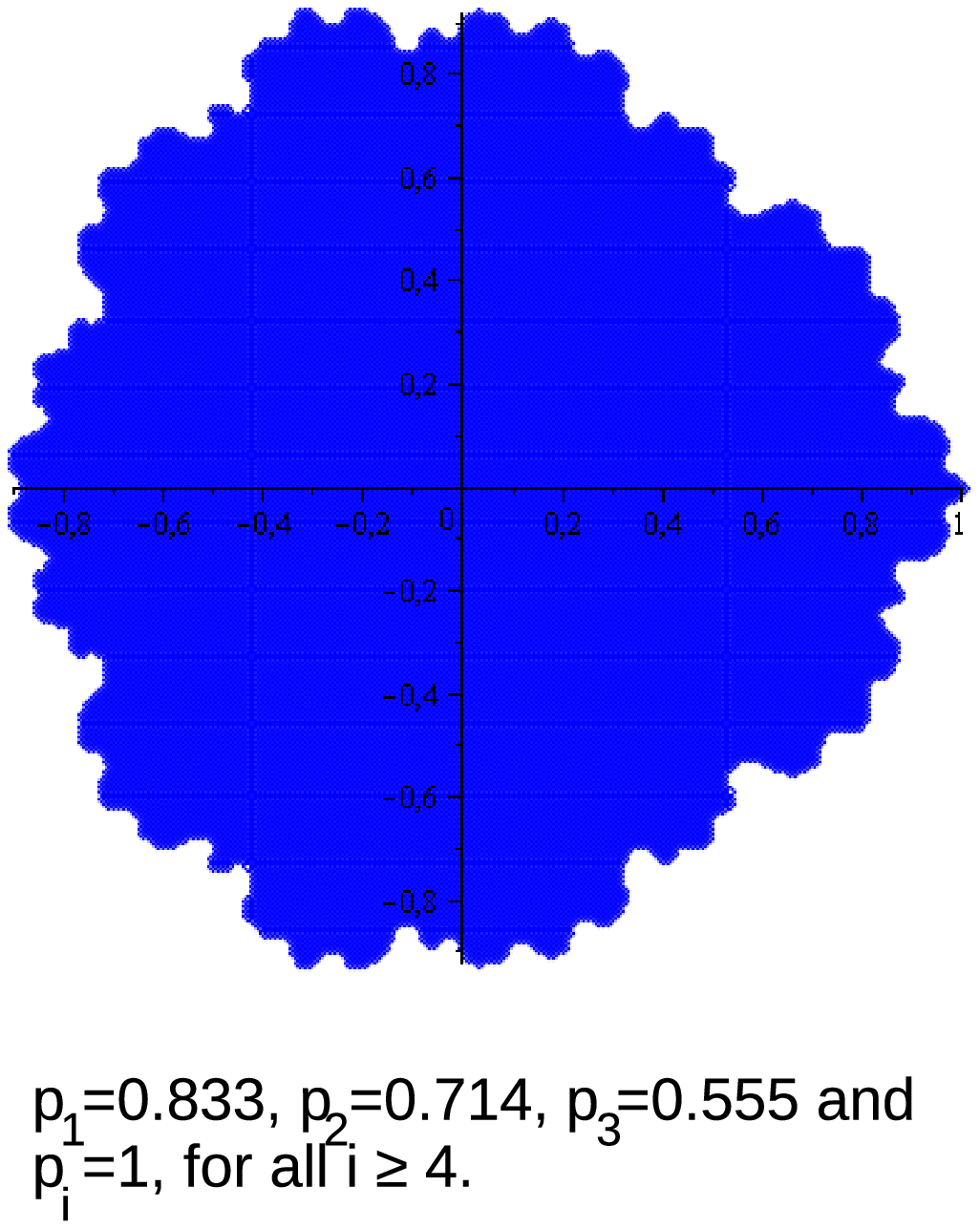}
\end{figure}
\end{Example}	

\begin{Example} 
	Consider the consecutive ordering Bratteli diagram $B$ represented by the matrix $M=\left(\begin{array}{cc} 1 & 5 \\ 9 & 2 \end{array}\right)$. Below, we present some pictures describing the set $\mathcal{E}=\mathcal{F}=\{\lambda\in\mathbb{C}: (u_{F_n}(\lambda))_{n\geq 0} \textrm{ is bounded}\}$ for some choices of $(p_i)_{i\ge 1}$.
	\begin{figure}[h!]
		\centering
		\includegraphics[scale=0.37]{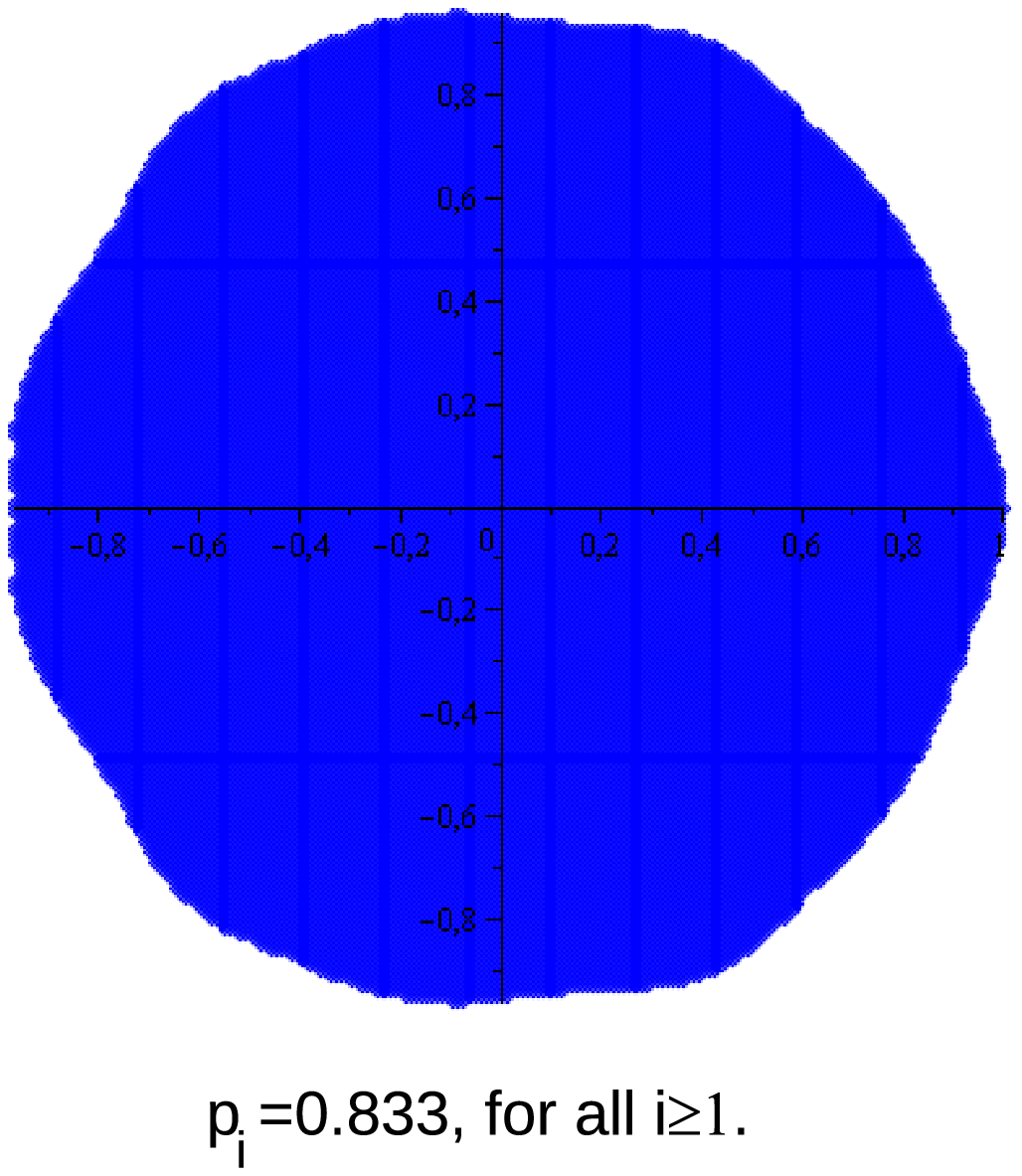}
		\includegraphics[scale=0.37]{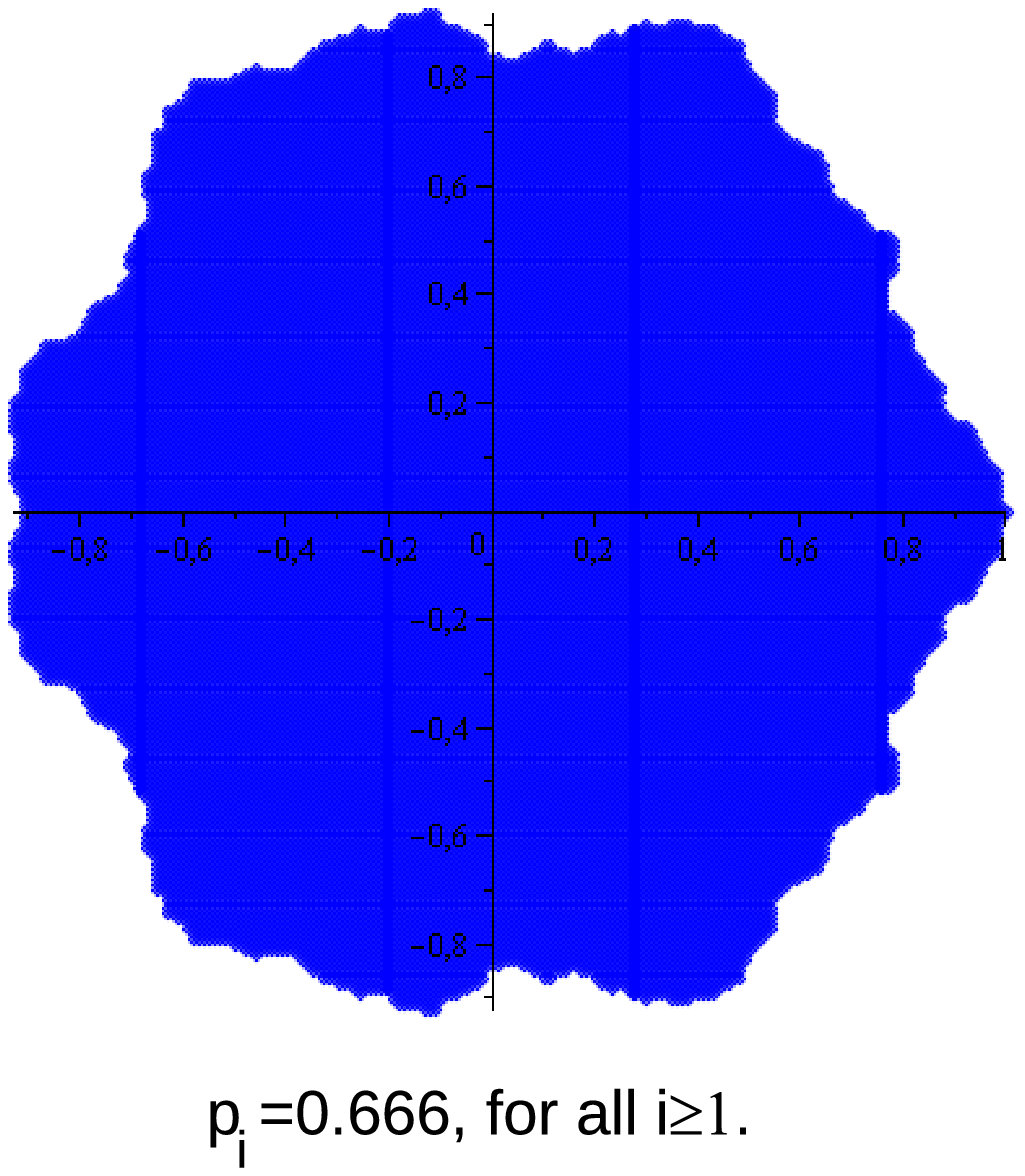}
		\includegraphics[scale=0.37]{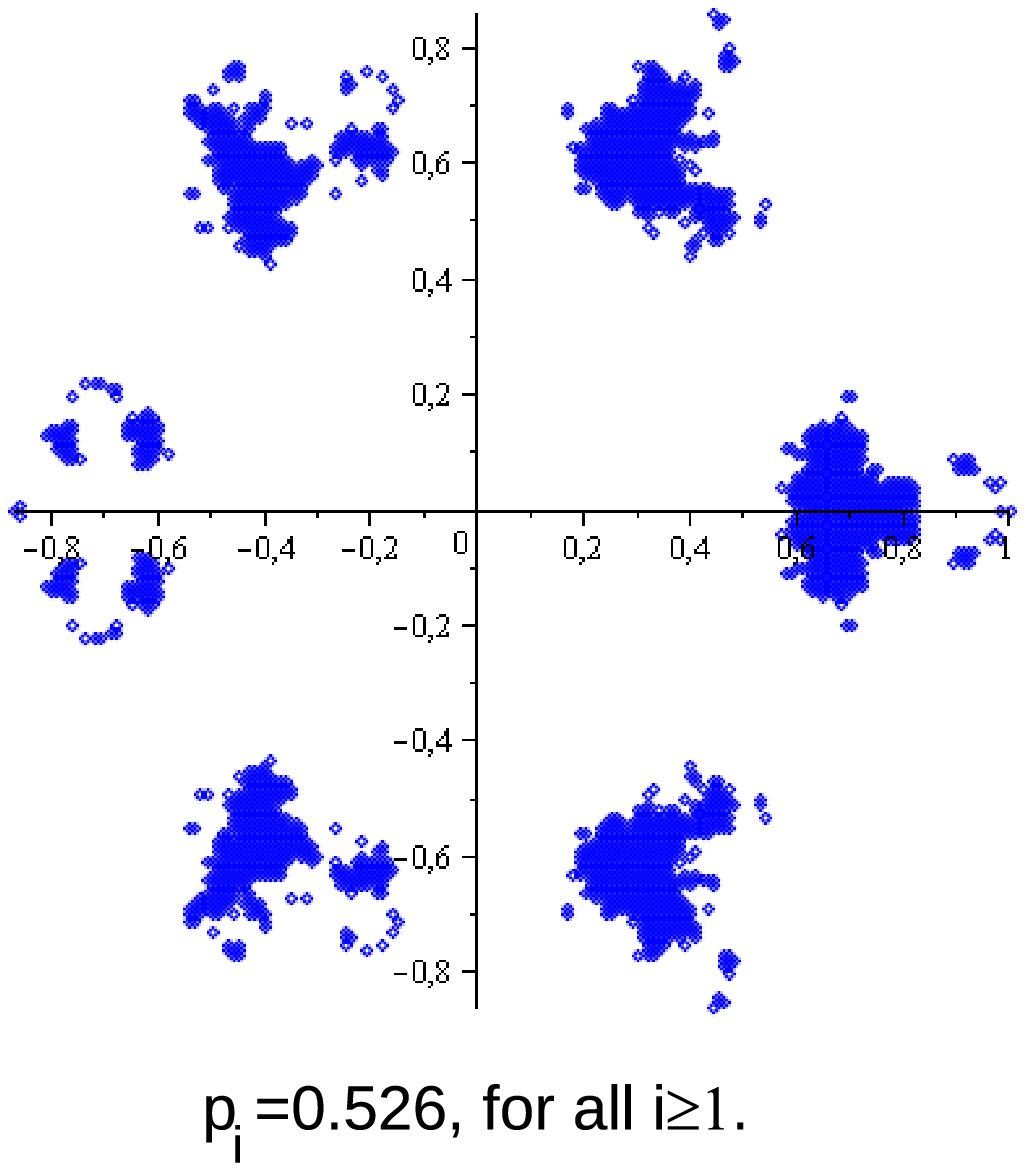}
	\end{figure}
\end{Example}

\begin{remark}
It will be interesting to compute the different parts of the spectrum of $S$ acting on other Banach spaces like $c_0$, $c$, $l^q(\mathbb{N})$, with $q\ge 1$ as done for base $2$ in \cite{em} and for Cantor systems of numeration in \cite{mv}.
\end{remark}

\subsection{Some topological properties of the set $\mathcal{E}$}

Let us suppose for simplicity that $p_i=p\in]0,1[$, for all $i\geq 1$.

\begin{Theorem} \label{teotopologico} Assume that   $\det M = ad-bc < 0$ and $bc>(ad-bc)^2$, then 
the set $\mathcal{E}$ satisfies the following properties:
\begin{enumerate}
	\item $\mathbb{C}\setminus \mathcal{E}$ is a connected set.
	\item If $p<\frac{1}{2}$, then $\mathcal{E}$ is not connected.
\end{enumerate}
\end{Theorem}

\begin{remark}
1. We conjecture that under the hypothesis of the last theorem, there exists $\frac{1}{2}<\delta <1$ such that $\mathcal{E}$ is connected, for all $p\geq \delta$.

\noindent 2. Recall that we are supposing $abc>1$, then the last theorem holds if $\det M= -1$.
 \end{remark}

\begin{Lemma} \label{lematopologico}
Under the hypothesis of Theorem \ref{teotopologico},  there exists a sufficiently large constant $R=R(p)>1$ such that 
\begin{center} 
$\mathcal{E}=\bigcap_{n=0}^{+\infty} u_{F_n}^{-1}\overline{D(0,R)}$, with $u_{F_{n+1}}^{-1}\overline{D(0,R)}\subset u_{F_n}^{-1}\overline{D(0,R)}$, for all $n\geq 0$.
\end{center} 
\end{Lemma}
\begin{proof}
Let $R>\frac{2-p}{p}>1$ be a constant that later will be chosen sufficiently large.

\smallskip

\noindent \textbf{Claim 1:} If $|u_{F_0}|=|w_{F_0}|>1$ then $(u_{F_k})_{k\geq 0}$ is not bounded.

The claim follows directly from \eqref{deffib} and we leave the details to the reader.

\smallskip

\noindent \textbf{Claim 2:} if $|u_{F_n}(\lambda)|>R$ for some integer $n\geq 0$, then $(u_{F_k})_{k\geq 0}$ is not bounded.

From Claim 1, we have that Claim 2 holds for $n=0$.

If $|u_{F_1}|> R$, then $|u_{F_0}|^{a+b}\geq pR-(1-p)$. Thus, $|u_{F_0}|>1 $ and so $(u_{F_k})_{k\geq 0}$ is not bounded. Hence, the claim is true for $n=1$.

Assume that the claim is true for all integers $k\in\{0,\ldots,n-1\}$.

Suppose that $|u_{F_n}|>R$. Hence, $|w_{F_n}|\leq 1$ since otherwise $(u_{F_k})_{k\geq 0}$ is unbounded.

Therefore, $$pR-(1-p)< |pu_{F_n}+1-p|=|u_{F_{n-1}}^aw_{F_{n-1}}^b|=|u_{F_{n-1}}^cw_{F_{n-1}}^d|^{\frac{a}{c}}|w_{F_{n-1}}|^{\frac{bc-ad}{c}} .$$

Since $u_{F_{n-1}}^cw_{F_{n-1}}^d=pw_{F_n}+1-p$ and $|w_{F_n}|\leq 1$, we deduce that $$(pR-(1-p))^{\frac{c}{bc-ad}}< |w_{F_{n-1}}|.$$
Thus $|w_{F_{n-1}}|>1$ and hence $|u_{F_{n-1}}| \leq 1$.

On the other hand, $$p(pR-(1-p))^{\frac{c}{bc-ad}}-(1-p)< |u_{F_{n-2}}^cw_{F_{n-2}}^d|\leq |pu_{F_{n-1}}+1-p|^{\frac{d}{b}}|u_{F_{n-2}}|^{\frac{bc-ad}{b}}\leq |u_{F_{n-2}}|^{\frac{bc-ad}{b}}.$$

Thus, $$|u_{F_{n-2}}|> \left(p(pR-(1-p))^{\frac{c}{bc-ad}}-(1-p)\right)^{\frac{b}{bc-ad}}\geq K(p) R^{\frac{bc}{(bc-ad)^2}},$$ where $K(p)$ is a positive constant.

Since $bc> (bc-ad)^2$, it follows that $|u_{F_{n-2}}|> R$, for $R$ sufficiently large and the proof of claim 1 is done.

Hence, by the claim, we deduce that $\mathcal{E}=\bigcap_{n=0}^{+\infty} u_{F_n}^{-1}\overline{D(0,R)}$.

\smallskip

\noindent \textbf{Claim 3:} $u_{F_{n+1}}^{-1}\overline{D(0,R)}\subset u_{F_n}^{-1}\overline{D(0,R)}$, for all $n\geq 0$.

Indeed, if $| u_{F_1}| = | \frac{1}{p} u_{F_{0}}^{a+b} - \frac{1-p}{p} | \leq R$, then
$$| u_{F_0} | \leq (p R+ 1-p)^{1/ a+b} <R,$$
and the claim holds for $n=0$.
The case $n=1$, can also be proved easily and is left to the reader.

Assume that the claim holds for all $k=0,\; \ldots, n-1,\; n \geq 2$ and that $| u_{F_{n+1}} | \leq R$.
Suppose that $| u_{F_{n}} | > R$, then $| w_{F_{n}} | \leq 1$, since otherwise $| u_{F_{n+1}} | >R$.
We deduce as done before that
$$(pR-(1-p))^{\frac{c}{bc-ad}}< |w_{F_{n-1}}|.$$
Hence
 $\frac{1}{p} | u_{F_{n-1}} |^c  (pR-(1-p))^{\frac{dc}{bc-ad}} - \frac{1-p}{p} \leq | w_{F_{n}} |  \leq 1$.
We deduce that $$|u_{F_{n-1}}|<(pR-1+p)^{\frac{-d}{bc-ad}}< 1 <R.$$

Thus as done before $$|u_{F_{n-2}}|> \left(p(pR-(1-p))^{\frac{c}{bc-ad}}-(1-p)\right)^{\frac{b}{bc-ad}}\geq K(p) R^{\frac{bc}{(bc-ad)^2}} >R.$$
This contradicts the hypothesis of induction for $k=n-2$ since $|u_{F_{n-1}}| <R$ and $|u_{F_{n-2}}|  >R.$

Hence $| u_{F_{n}} |  < R$ and we obtain the claim for $k=n$.
\end{proof}

\begin{Remark} \label{disco}
Lemma \ref{lematopologico} is true if we change $\overline{D(0,R)}$ by $D(0,R)$.
\end{Remark}

\begin{proof}[Proof of Theorem \ref{teotopologico}] 

$(1)$ By Lemma \ref{lematopologico}, $$\mathbb{C}\setminus\mathcal{E}=\bigcup_{n=0}^{\infty} \mathbb{C}\setminus u_{F_n}^{-1} \overline{D(0,R)}.$$
		Since $\mathbb{C}\setminus\overline{D(0,R)}$ is
	connected, it follows from the \emph{maximum modulus principle} that for
	each holomorphic map $u_{F_n}$, $\mathbb{C}\setminus
	u_{F_n}^{-1}\overline{D(0,R)}$ is connected for all $n\geq 0$. On the
	other hand, since $\mathbb{C}\setminus u_{F_n}^{-1}\overline{D(0,R)}$
	contains a neighbourhood of infinity for all $n\geq 0$, we deduce
	that $\mathbb{C}\setminus \mathcal{E}$ is connected.
	
\smallskip

$(2)$ We can show easily by induction on $n$ that $\frac{1-p}{p}$ is a critical point of $u_{F_n}$ (and also $w_{F_n}$), for all $n\geq 1$. Since by Remark \ref{remark417}, $\mathcal{E}\subset \sigma_a \subset \sigma$ and the spectrum $\sigma$ is contained in $\overline{D(0,1)}$, we deduce that if $p<\frac{1}{2}$, then $\frac{1-p}{p}\not\in \mathcal{E}$. Therefore, by Lemma \ref{lematopologico} and Remark \ref{disco}, there exists an integer $N$ such that $$\frac{1-p}{p}\not\in u_{F_n}^{-1} D(0,R) \textrm{, for all } n\geq N.$$ Hence, by the Riemann-Hurwitz formula, we deduce that $u_{F_n}^{-1} D(0,R)$ is not connected, for all $n\geq N$. Thus, by Lemma \ref{lematopologico} we are done.
\end{proof}

%%%%%%%%%%%%%%%%%%%%%%%%%%%%%%%%%%%%%%%%%%%%%%%%%%%%%%%%%%%%%%%%%%%%%%%%%%%%%%%%%%%%%%%%%%%%%%%%%%
%%%%%%%%%%%%%%%%%%%%%%%%%%%%%%%%%%%%%%%%%%%%%%%%%%%%%%%%%%%%%%%%%%%%%%%%%%%%%%%%%%%%%%%%%%%%%%%%%%
%%%%%%%%%%%%%%%%%%%%%%%%%%%%%%%%%%%%%%%%%%%%%%%%%%%%%%%%%%%%%%%%%%%%%%%%%%%%%%%%%%%%%%%%%%%%%%%%%%

\section{Generalization}
\label{generalization}
Now, let $B=(V,E,\geq)$ be a $2 \times 2$ simple ordered Bratteli diagram endowed with the consecutive ordering and having incidence matrices $M_n=\left(\begin{array}{cc}
a_n & b_n \\
c_n & d_n
\end{array}\right)$. Suppose that $a_nb_nc_n>0$ and that $c_n+d_n>1$, for all $n\geq 1$. Therefore B satisfies Hypothesis A.

Consider $\left(\begin{array}{cccc}
A_k & B_k \\
C_k & D_k
\end{array}\right):= M_k\cdot M_{k-1}\cdot\ldots\cdot M_1$, for all $k\geq 1$.

Let $F_0=G_0=1$ and for each $k\geq 1$, let $F_k=A_k+B_k$ and $G_k=C_k+D_k$.

As before, we can prove that for all $n\geq 1$
$$
F_{n+1}=\left(a_{n+1}+\frac{b_{n+1}}{b_n}d_n\right)F_n-\left(\frac{b_{n+1}}{b_n}a_nd_n-b_{n+1}c_n\right)F_{n-1}
$$ 
and 
$$
G_{n+1}=\left(a_{n+1}+\frac{b_{n+1}}{b_n}d_n\right)G_n-\left(\frac{b_{n+1}}{b_n}a_nd_n-b_{n+1}c_n\right)G_{n-1} \, .
$$

Furthermore, for each $\lambda\in\mathbb{C}$, let $(u_n)_{n\geq 1}$ and $(w_{F_n})_{n\geq 0}$ be the sequences defined like in relation \ref{deffib}, changing $a$ and $b$ by $a_n$ and $b_n$ in $u_{F_n}$ and changing $c$ by $c_n$ and $d$ by $d_n$ in $w_{F_n}$, respectively.

Like in Theorem \ref{teorema431}, we can prove that $$\sigma_{pt}(S)=\{\lambda\in\mathbb{C}:(u_n(\lambda))_{n\geq 1}
\textrm{ is bounded}\},$$
where here, the maps $g_n:\mathbb{C}^2\longrightarrow\mathbb{C}^2$ are defined by 
$$g_n(x,y)=\left(\frac{1}{p_{n+1}}x^{a_n}y^{b_n}-\frac{1-p_{n+1}}{p_{n+1}}, \frac{1}{p_{n+1}}x^{c_n}y^{d_n}-\frac{1-p_{n+1}}{p_{n+1}}\right),$$ for all $n\geq 1$.

\medskip

Let $B=(V,E,\geq)$ be a $l \times l$ ($l\geq 3$) simple ordered Bratteli diagram endowed with the consecutive ordering and having incidence matrices $M_n=(m_{i,j}^{(n)})$ with $\sum_{j=1}^l m_{i,j}^{(n)}>1$, for all $i\in\{1,\ldots, l\}$ and $n\geq 1$.

Like before, we can prove that the point spectrum of $S$ is contained in the fibered Julia set $\{z\in\mathbb{C}^l: (\psi_n(z))_{n\geq 0} \textrm{ is bounded}\}$, where for all $z=(z_1,\ldots,z_l)\in\mathbb{C}^l$, $\psi_n(z)=g_n\circ\ldots\circ g_0 (z)$ and  $g_n:\mathbb{C}^l\longrightarrow\mathbb{C}^l$ are maps defined by
\begin{center}	
$g_n(z)=\left(
\frac{1}{p_{n+1}}z_1^{m_{1,1}^{(n)}}z_2^{m_{1,2}^{(n)}}\ldots z_l^{m_{1,l}^{(n)}} -\frac{1-p_{n+1}}{p_{n+1}},
\ldots,
\frac{1}{p_{n+1}}z_1^{m_{l,1}^{(n)}}z_2^{m_{l,2}^{(n)}}\ldots z_l^{m_{l,l}^{(n)}} -\frac{1-p_{n+1}}{p_{n+1}} \right)$
\end{center}
 for all $n\geq 1$.

%%%%%%%%%%%%%%%%%%%%%%%%%%%%%%%%%%%%%%%%%%%%%%%%%%%%%%%%%%%%%%%%%%%%%%%%%%%%%%%%%%%%%%
%%%%%%%%%%%%%%%%%%%%%%%%%%%%%%%%%%%%%%%%%%%%%%%%%%%%%%%%%%%%%%%%%%%%%%%%%%%%%%%%%%%%%%
%%%%%%%%%%%%%%%%%%%%%%%%%%%%%%%%%%%%%%%%%%%%%%%%%%%%%%%%%%%%%%%%%%%%%%%%%%%%%%%%%%%%%%

\section*{Acknowledgments} The second author thanks Fabien Durand and Thierry Giordano for fruitful discussions.

% \bibliographystyle{alpha}
% \bibliography{bib}

\end{document}